\documentclass[12pt,reqno]{amsart}
\usepackage[comma,numbers]{natbib}
\usepackage[a4paper,left=1in,right=1in,top=1.25in,bottom=1.25in]{geometry}
\usepackage{color}
\usepackage{graphicx}

\newtheorem{theorem}{Theorem}[section]
\newtheorem{corollary}[theorem]{Corollary}
\newtheorem{lemma}[theorem]{Lemma}
\theoremstyle{definition}
\theoremstyle{remark}

\def\numberlikeadb{\global\def\theequation{\thesection.\arabic{equation}}}
\numberlikeadb

\newcommand{\eqa}{\begin{eqnarray}}
\newcommand{\ena}{\end{eqnarray}}
\newcommand{\eq}{\begin{equation}}
\newcommand{\en}{\end{equation}}
\newcommand{\eqs}{\begin{eqnarray*}}
\newcommand{\ens}{\end{eqnarray*}}

\def\Eq{\ =\ }
\def\Def{\ :=\ }
\def\Le{\ \le\ }
\def\pr{\mathbb{P}}
\def\ex{\mathbb{E}}
\def\ignore#1{}

\def\a{\alpha}

\def\s{\sigma}

\def\Ref#1{(\ref{#1})}

\def\re{\mathbb{R}}

\def\xx{{\mathcal X}}

\def\th{\theta}

\def\b{\beta}
\def\d{\delta}
\def\n{\nu}
\def\l{\lambda}

\def\eth{\eta_{\Theta}}

\def\g{\gamma}

\def\b{\beta}

\def\t{\tau}

\def\cmax{c_{\rm max}}
\def\amax{a_{\rm max}}
\def\smax{\s_{\rm max}}
\def\rmax{\rho_{\rm max}}
\def\emax{e_{\rm max}}
\def\emin{e_{\rm min}}
\def\amin{a_{\rm min}}
\def\smin{\s_{\rm min}}
\def\cth{{c}_{\Theta}}
\def\aot{{\alpha_1,\alpha_2}}
\def\PEV{\l}
\def\ddd{t}

\def\QQ{{\mathcal Q}}

\def\kd{v_{d-1}}
\def\iid{i.i.d.}
\def\te{{\tilde e}}
\def\tE{{\widetilde E}}

\def\ignore#1{}

\begin{document}
\title[]{Local approximation of a metapopulation's equilibrium}
\maketitle
\noindent A.D. BARBOUR, R. McVINISH and P.K. POLLETT  \footnote{ADB is supported in part by
Australian Research Council (Discovery Grants DP150101459 and DP150103588). PKP and RM are supported in 
part by the Australian Research Council (Discovery Grant DP150101459 and the ARC Centre of Excellence for 
Mathematical and Statistical Frontiers, CE140100049). We thank the referees for their helpful comments. }\\
Universit\"at Z\"urich and University of Queensland\\

% ----------------------------------------------------------------

\noindent ABSTRACT. 
We consider the approximation of the equilibrium of a metapopulation model, in which a finite number of patches are randomly distributed over a bounded subset~$\Omega$ of Euclidean space. The approximation is good when a large number of patches contribute to the colonization pressure on any given unoccupied patch, and when the quality of the patches varies little over the length scale determined by the colonization radius. If this is the case, the equilibrium probability of a patch at~$z$ being occupied is shown to be close to~$q_1(z)$, the equilibrium occupation probability in Levins's model, at any point $z \in \Omega$ not too close to the boundary, if the local colonization pressure and extinction rates appropriate to~$z$ are assumed.  The approximation is justified by giving explicit upper and lower bounds for the occupation probabilities, expressed in terms of the model parameters.  Since the patches are distributed randomly, the occupation probabilities are also random, and we complement our bounds with explicit bounds on the probability that they are satisfied at all patches simultaneously.

\bigskip


\section{Introduction}
\setcounter{equation}{0}

A number of papers have addressed the problem of approximating a complex metapopulation model by Levins's 
model \citep[][among others]{Etienne:02,Keeling:02,BP:04,OC:06}. For example, in the setting of \citet{OC:06},  a 
metapopulation is taken to consist of an infinite number of patches in~$\mathbb{R}^{d} $. In their simplest case, 
the locations of the patches are determined by a Poisson point process with constant intensity measure; 
they also consider stationary point processes with spatial correlation.  The colonization rate of an empty 
patch is determined by its distance from the occupied patches and by a colonization kernel. Under such 
conditions, \citet{OC:06} give asymptotic expansions describing the differences between the equilibrium 
properties of the metapopulation and  what would be expected under a uniform mean field Levins model, in the 
limit where the range of the colonization kernel tends to infinity. Their expansions were formally
justified in \cite{OvasEtAl:14}.

In this paper, we take a somewhat different approach.  First, we are interested in metapopulations which are not 
infinite in extent, but consist of only finitely many patches, and whose underlying landscape is not uniform; 
in particular, it might consist of a number of regions in which the metapopulation is viable, separated by 
regions where it is not. In previous work \cite{BMP:15}, we have demonstrated that deterministic metapopulation 
models provide good approximations to their stochastic counterparts, at least over finite time horizons, 
provided that the colonization pressure at a patch results from the sum of the effects of many other patches 
--- this is equivalent to the assumption in \citet{OC:06} that the colonization kernel has long range. However, 
if the landscape is not uniform, even the equilibrium state of the deterministic system is unknown. In this 
paper, we show how to construct an approximation to  the equilibrium state of the deterministic system, 
provided that the properties of the landscape
do not vary much over the range of the colonization kernel. The approximation is local, in the sense that 
the equilibrium probability of a patch at position~$z$ being occupied is what it would be if the landscape 
were  everywhere constant, with its parameters taking the values that are taken at~$z$.  
Rather than justifying the approximation in terms of limit theorems, we prefer to give explicit bounds on 
the accuracy of the approximation, which depend on the expected number of patches 
contributing to the colonization pressure at a given patch, on the possible variation of the landscape 
within the colonization radius, and on the ratio of the colonization radius to the diameter of the entire region 
--- in a finite region, boundary effects also play a part.  Patches are modelled as the points of a 
Bernoulli point 
process with spatially varying intensity, and so such error bounds cannot be definitive;  instead, we also give 
expressions bounding the probability that our error bounds are correct.

\section{The equilibrium of a metapopulation}
\setcounter{equation}{0}

The incidence function model of \citet{Hanski:94} in~$d$ dimensions for a metapopulation comprising $ n $~patches 
spread over a habitat~$\Omega$ of volume~$A$ is a discrete time Markov chain on $\xx := \{0,1\}^{n} $. Usually, $d=2$, 
and we think of volume as an area, but this is not needed here.  Denote this Markov chain by 
$ X_{t}= (X_{1,t},\ldots,X_{n,t}) $, where $ X_{i,t} = 1 $ if patch~$i$ is occupied at time $ t $ and $ X_{i,t} = 0 $ 
otherwise. We assume that the patch size and its ability to support a local population depend only on the patch location. 
Let $z_{i} \in \Omega \subset \mathbb{R}^{d} $ denote the location of the $i$-th patch. The transition probabilities 
of the Markov chain are determined by how well the patches are connected to each other and by the probability of local 
extinction. Define the functions $S_i\colon\, [0,1]^{n} \mapsto [0,\infty) $, which represent the aggregate migration 
pressure on patch~$i$ from the remaining patches, by
\begin{equation}
   S_{i}(x) \Eq \frac{A}{(n-1)} \sum_{j\neq i}  a(z_{j}) c(z_{i},z_{j};r) x_{j},  \label{IFM:Connect}
\end{equation}
where 
\[
    c(z,y;r) \Def r^{-d} c_{z}(\|z-y\|/r),
\]
and, for each~$z$, $c_z$ is an integrable function with maximum value at most~$\cmax$.  In what follows, we assume 
that $c_{z}(x) = 0$ for $x > 1$, so that the migration range is bounded by~$r$; this is to simplify the
analysis, and could be relaxed. The average density of population is given 
by the ratio $n/A$, so that, within such a range, there can be expected to be about $nr^d/A$ patches, over which the 
migration effort of a patch~$j$ is distributed.  Hence each patch contributes about $(nr^d/A)^{-1}$ of its effort to 
each other neighbouring patch, and this is why the ratio $A/(n-1)$ appears in~$S_i$, and the normalization~$r^{-d}$ in 
the definition of~$c(z,y;r)$.
\ignore{We shall suppose also that $a(z) \le \amax$ for all~$z \in \Omega$. }
Conditional on $ X_{t}$ and the set of patch 
locations $ \{z_{i}\}_{i=1}^{n} $, the $ X_{i,t+1} \ (i=1,\ldots,n) $ are 
independent with transition probabilities
\begin{equation}
   \pr\left(X_{i,t+1}=1 \ \middle| \ X_{t} , z_{1},\ldots,z_{n}\right)  \Eq 
      f(S_{i}(X_{t}))\left(1-X_{i,t}\right) +  (1-e(z_{i}))X_{i,t}. \label{Eq1}
\end{equation}
If patch $ i $ is occupied at time $ t $, then that population survives to time $ t+1 $ with probability $ 1-e(z_{i}) $. 
Otherwise, it is colonised with probability $  f(S_{i}(X_{t}))$. 

\citet[section 6.3]{AM:02} proposed a continuous time analogue of the incidence function model.   Their model is a 
continuous time Markov chain $ X(t) = (X_{1}(t),\ldots,X_{n}(t)) $ on $ \xx$, whose transition rates in the above 
setting are given by
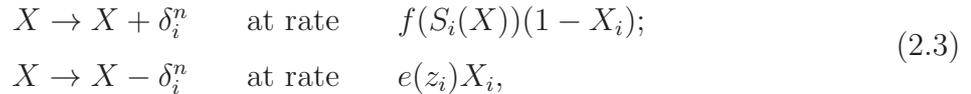
\begin{eqnarray}
  \begin{array}{rcl}
  X \rightarrow X + \delta_{i}^{n} & \quad \mbox{at rate} \quad & f(S_{i}(X)) (1-X_{i});  \\
  X \rightarrow X - \delta_{i}^{n} & \quad \mbox{at rate} \quad & e(z_{i}) X_{i}, 
  \end{array} \label{ADB-Eq4}
\end{eqnarray}
and $ \delta^{n}_{i}$  is the vector of length $ n $ with $ 1 $ at position $ i $ and zeros elsewhere. 

Since both processes are finite state Markov processes, with the extinction state absorbing and accessible 
from all other states, the extinction state is almost surely eventually reached. However, in many circumstances,
the processes may remain for long periods in an apparent stochastic equilibrium, a quasi-stationary distribution.  
In~\cite{BMP:15}, it is shown that the stochastic processes can be  well approximated by corresponding 
deterministic systems, at least over bounded time intervals.  Thus, if the stochastic processes have initial 
conditions corresponding to any equilibrium of the deterministic systems, they remain close  to this equilibrium 
over bounded time intervals, with asymptotically high probability. In this paper, working under conditions which 
guarantee at most one equilibrium of the deterministic systems other than extinction, we address the problem of 
describing the equilibrium.

The deterministic approximation of the Markov chain defined by \Ref{Eq1} was proposed by \citet{OH:01}. 
Let $p_{i,t}$ be the probability that patch~$i$ is occupied at time~$t$ and let $ p_t = (p_{1,t},\ldots,p_{n,t}) $. 
As in the incidence function model, they model the change in $p_t$ by
\begin{equation}
   p_{i,t+1}-p_{i,t} \Eq f(S_{i}(p_t))(1-p_{i,t}) - e(z_i) p_{i,t}. \label{Eq2}
\end{equation}
For the continuous time metapopulation model (\ref{Eq2}), the deterministic approximation is provided by the 
spatially realistic Levins model \cite{HG:97}. This model is a system of ordinary differential equations 
\begin{equation}
   \frac{dp_{i}(t)}{dt} \Eq f(S_{i}(p(t)))(1-p_{i}(t)) - e(z_{i}) p_{i}(t), \label{Eq3}
\end{equation}
for $p\colon \re_+ \to [0,1]^n$. The equilibrium levels of both these deterministic models are given by the 
fixed points~$p_n^*$ of the function $ E_{n}: [0,1]^{n} \rightarrow [0,1]^{n} $, where
\[
    E_{n}(p)_{i} \Def 
        \frac{f\left((A/(n-1)) \sum_{j\neq i} a(z_{j}) c(z_{i},z_{j};r) p_{j}\right)}
             { e(z_{i}) + f\left((A/(n-1)) \sum_{j\neq i} a(z_{j}) c(z_{i},z_{j};r) p_{j}\right)}.
\]
Define the matrix $ T_{ij} = f^{\prime}(0) (A/(n-1)) a(z_{j}) c(z_{i};z_{j};r) /e(z_{i}) $ 
for $ i \neq j $ and $ T_{ii} = 0 $, and let $ \PEV(T) $ be the 
Frobenius-Perron eigenvalue of~$T$. When $ f $ is continuous and concave and $ T $ is primitive, that is, if $ T^{k} $ is a positive matrix for some $ k $, the cone limit set trichotomy \cite[Theorem 6.3]{HS:05} can be applied to conclude that 
\begin{itemize}
\item If $ \PEV(T) \leq 1 $, then $0 \in \re^n$ is the only fixed point of $ E_{n} $;
\item If $ \PEV(T) > 1 $, then, in addition to $0$, $ E_{n} $ has a non-zero fixed point.
\end{itemize}
In what follows,  we denote the largest fixed point of $ E_{n} $ by~$ p^{\ast}_{n} $. Our aim is to 
determine good approximations to $ p_{n}^{\ast} $.  We do so under certain assumptions. 
\begin{itemize}
  \item[(A)] Independent patch locations: The patch locations $ z_{i} $ are independently distributed over the closed, bounded and
{\it connected\/} set~$\Omega$, with 
   probability density $A^{-1} \sigma(\cdot) $.
\end{itemize}
We then define
\eq\label{rho-def}
   \rho(z) \Def \int_\Omega a(y)c(z,y;r) \sigma(y)\,dy.
\en
\begin{itemize}
\item[(B)] Bounded support of colonisation kernel: For all $  z \in \Omega $, $ c_{z}(x) > 0 $ for 
all $ x \in [0,1) $ and $ c_z(x) = 0 $ for all $ x > 1 $.
\item[(C)] Smoothness: The functions $ e, a, f, \sigma $ and $ \rho $  are continuously differentiable 
on $ \Omega $. Their respective Lipschitz constants are denoted by $ L_{e}, L_a, L_f, L_\sigma $ 
and~$ L_\rho $. The colonisation kernel $ c_{z}(x) $ is uniformly continuous on $ \Omega\times[0,1)$.
\item[(D)] Bounded:  The functions $ e,\sigma$ and~$a$  are bounded above by $\emax,\smax$ and~$\amax$ and below by
{\it positive\/} constants $\emin,\smin$ and~$\amin$, respectively. The function $ c_{z} $ is bounded above by $\cmax $, 
for each $z\in\Omega$.
\item[(E)] Concave colonisation function: $ f $ is an increasing concave function such that $ f(0) = 0 $,
and $f(0) > 0$. 
\end{itemize}
Note that, from Assumption~E, there is a constant~$C_1 > 0$ such that $f(x) \ge f'(0)x - C_1x^2$ for all $x \ge 0$, 
and that $L_f = f'(0)$.

The quantities in the set $\QQ := \{L_f, \smax,\amax,\emax,\smin,\amin, \emin, \cmax\}$ can all be taken to be fixed without reference to the values of $n,r$ and~$A$. However, the Lipschitz constants $L_{e}, L_a, L_\sigma$ and~$L_\rho$ measure the maximum possible changes in the corresponding functions per unit displacement of the arguments, and have units $\{\mbox{distance}\}^{-1}$.  Correspondingly, we shall take $rL_{e}, rL_a, rL_\sigma$ and~$rL_\rho$ rather than $L_e,L_a,L_\s$ and~$L_\rho$ to be the quantities of biological interest. The error in our approximations is measured in terms of these quantities;  in particular, certain inequalities between them must be satisfied, if our approximation bounds are to be valid.  However, we shall tacitly think of~$n$ and the combination $nr^d/A$ being big, and of $r^d/A, rL_{e}, rL_a, rL_\sigma$ and~$rL_\rho$ being small, all of which are needed if our approximation errors are to be small.  In essence, we establish conditions under which the probability of a patch at location~$z$ being occupied is the same as would be the case if the environment were locally constant, with the values taken at~$z$, and if the patches were not discrete, but were smeared over
the habitat according to the density function~$\s$.  Thus we want $rL_{e}, rL_a, rL_\sigma$ and~$rL_\rho$ to be small (within the colonization radius, environmental conditions do not vary a lot), and~$nr^d/A$ to be big (many patches averaging to realize the colonization pressure).   The requirement that $r^d/A$ be small is needed to prevent boundary effects dominating the final result.

\section{Approximation of the equilibrium}\label{Sect3}
\setcounter{equation}{0}

To construct our approximation, we first define the function $ q_{\alpha}\colon\, \Omega \rightarrow [0,1] $ such that $ q_{\alpha}(z) $ is the largest solution to the equation
\begin{equation}
x = \frac{f(x\rho(z))}{\alpha e(z) + f(x\rho(z))}. \label{Eq-rev1}
\end{equation}
This equation always has the solution $ x = 0 $ by Assumption E so $ q_{\alpha} $ is well defined. If $ f^{\prime}(0) \rho(z) > \alpha e(z) $, then equation (\ref{Eq-rev1}) also has a non-zero solution. We note that $ q_{\alpha} $ is decreasing in $ \alpha $ since the right-hand side of equation (\ref{Eq-rev1}) is decreasing in $ \alpha $. When $ f(x) = x $, $ q_{\alpha} $ can be written explicitly as
\[
q_{\alpha}(z) = 1 - \frac{\alpha e(z)}{\rho(z)}.
\]
We would ideally like to show that $ q_1(z_i)$ is a good approximation to $p_i$.  To do so, we establish upper and lower bounds, one using $q_\a(z_i)$, with $\a$ less than, but close to, $1$, and the other using $q_\b(z_i)$, for~$\b$ close to, 
but larger than, $1$.  More precisely, we first show that, with high probability, the function $ p_{\aot}^{+}\colon\, \Omega \rightarrow [0,1] $, defined as
\[
   p_{\aot}^{+}(z) \Def q_{\alpha_{1}}(z) \vee (1-\alpha_{2}),
\]
provides an upper bound on $ p^{\ast}_{n} $ for some $ \alpha_{1}$ and $\alpha_{2} $ such that 
$ 1/2 < \alpha_{2} \leq \alpha_{1} < 1 $; that is, that $ p^{\ast}_{i,n} \leq p^{+}_{\aot}(z_{i}) $ for all $ i $.  
To construct a lower bound on $ p^{\ast}_{n} $, we then choose some $ \Theta  \subset  \Omega $ with a smooth 
boundary $ \partial \Theta $. For $ m > 0 $ and $ \beta > 1 $, we define the function 
$p^{-}_{\Theta,\beta,m} \colon\, \Theta \rightarrow [0,1] $ by
\[
   p^{-}_{\Theta,\beta,m}(z) \Def \left( m\|z-\partial \Theta\| \wedge q_{\beta}(z) \right),
\]
where $ \| z - \partial \Theta\| $ is the distance from $ z $ to the boundary of $ \Theta $. The aim is then to 
find choices of $\beta $ and~$ m $  such that $ p^{-}_{\Theta,\beta,m} $ provides a lower bound on $ p^{\ast}_{n} $ 
with high probability.

\begin{theorem}[Upper bound] \label{thm:UB}
Let $ N(\Omega,r) $ is the number of balls of radius $ r $ required to cover $ \Omega $. Suppose that 
Assumptions A--E hold and that $ n > 2 N(\Omega,r/3) $.  Define 
\eq\label{rho-max-def}
   \eta_\Omega \Def \min_{z\in\Omega} q_{1}(z); \quad \rmax \Def \amax\cmax\smax v_d,
\en
where $ v_{d} $ is the volume of the $d$-dimensional unit ball. Assume that 
\begin{align}
  2L_{f}L_q r \rmax &\Le \emin (1-\alpha_{1})(\alpha_{1} \eta_\Omega \vee(1-\alpha_{2})); \label{UB:ineq1} 
\end{align}
where 
\begin{equation}
    L_q \Def \frac{\sqrt d}{\emin} \left( 2L_f L_{\rho}+ L_{e}\right). \label{lem:dif:eq4}
\end{equation}
Then, if $L_f\rmax/\emin > 1/2$,
\begin{align*}
   \mathbb{P} \bigl(  p_{n,i}^{\ast} \leq p^{+}_{\aot}(z_{i}) & \mbox{ for all } i=1,\ldots,n \bigr) \
        \geq\ 1 - 2n \exp\left( -C_{2} \frac{(n-1) r^d}{A}\,
             \frac{ \emin^{2} (1-\alpha_{1})^{2}}{16 \amax^2 L_{f}^{2}} \right) \\
   &   - \frac{n}2\, 
           \exp\left( - n \min_{z\in\Omega} A^{-1}\int_{\Omega} \mathbb{I}(\|y-z\|\leq r/3) \sigma(y)\, dy\right),
\end{align*}
where
\begin{align*}
  C_{2} \Def \{3\{\cmax\}^2 \smax v_d\}^{-1}.
\end{align*}
If $L_f\rmax/\emin \le 1/2$, then
\begin{align*}
   \mathbb{P} \left(  p_{n,i}^{\ast} =0 \mbox{ for all } i=1,\ldots,n\right) 
          &  \geq\ 1 - 2n \exp\left\{ -C_{2} \frac{(n-1) r^d}{A} \Bigl(\frac{ \rmax}{2 \amax}\Bigr)^2 \right\} \\
   &   - \frac{n}2\, 
           \exp\left( - n \min_{z\in\Omega} A^{-1}\int_{\Omega} \mathbb{I}(\|y-z\|\leq r/3) \sigma(y)\, dy\right).
\end{align*}
\end{theorem}

The quantity $\eta_\Omega$, which acts as an indicator of the viability of the least favourable patch in~$\Omega$, plays an important role in restricting the range of $ \alpha_{1} $ and $ \alpha_{2} $ through inequality~(\ref{UB:ineq1}). When all patches are locally viable, that is when the colonization rate is greater than the extinction rate at each patch, $ q_{1}(z) > 0 $ for all $ z \in \Omega $ and $ \eta_{\Omega} > 0 $. In this case, inequality~(\ref{UB:ineq1}) in Theorem~\ref{thm:UB} can be satisfied by taking $ \alpha_{2} = \alpha_{1} $ and $ 1- \alpha_{1} $ as big as a multiple of~$L_q r$. On the other hand, if some of the patches in $\Omega$ act like population `sinks', having colonization rates less than their extinction rates, then $ \eta_\Omega = 0 $, and we must take $ (1-\alpha_{1})(1-\alpha_{2}) $ to be bigger than a multiple of $ L_q r$. This forces a trade-off in the tightness of the upper bound between different regions of $ \Omega $ as larger $ \alpha_{1} $ tightens the upper bound where the equilibrium is moderate/large and larger $ \alpha_{2} $ tightens the upper bound where the equilibrium is small. We can satisfy inequality~(\ref{UB:ineq1}) by taking $ \alpha_{1} = \alpha_{2} $ with $ 1-\alpha_{1} $ now a multiple of $ (L_{q} r)^{1/2} $. With this choice we attempt to make the upper bound as tight as possible uniformly over all of $ \Omega $.

Our result shows that if $ L_{f} \rmax/\emin \le 1/2 $, then the extinction state is the only equilibrium with high probability. This condition seems stronger than necessary and one might expect that it would be sufficient that $ L_{f} \max_{z\in \Omega} \frac{\rho(z)}{e(z)} < 1 $ for the extinction state to be the unique equilibrium with high probability. Unfortunately, a finer analysis would be required to establish such a result. Assuming $ L_{f} \max_{z\in \Omega} \frac{\rho(z)}{e(z)} < 1 $, Theorem~\ref{thm:UB} can only establish that any equilibrium is bounded by a multiple of $ (L_{q}r)^{1/2} $ with high probability provided $ L_{q} r $ is sufficiently small and  $ nr^{d+1} $ is large.

For sufficiently regular regions $ \Omega $, $ N(\Omega,r/3) \leq c Ar^{-d} $ for some constant $ c >0$, and 
so the assumption that $ n > 2N(\Omega,r/3) $ would normally be satisfied in situations where we expect the bound to hold 
with high probability.  Finally, defining $B_x(t) \Def \{ z: \|z - x \| \leq t \}$, we note that if $ \Omega $ is $r$-smooth, in the sense that, for some $\Omega' \subset \Omega$ ,
\[
     \Omega \Eq \bigcup_{x \in \Omega'} B_x(t_x),\qquad\mbox{with}\quad t_x \ge r \mbox{ for all } x\in\Omega',
\]
then
\[
    \int_{\Omega} \mathbb{I}(\|y-z\|\leq r/3) \, dy \ \ge\ c_d r^d \quad \mbox{for all } z\in\Omega,
\]
where~$c_d$ is a geometric constant depending only on~$d$.
In such circumstances,
\eq\label{r-smooth}
     n\min_{z\in\Omega} A^{-1}\int_{\Omega} \mathbb{I}(\|y-z\|\leq r/3) \sigma(y)\, dy\ \geq\ c_d\smin (nr^{d}/A).
\en

The lower bound is more complicated to state because boundary effects make themselves felt.  We restrict 
ourselves to proving lower bounds for $p_{i,n}^{\ast}$ for points~$z_i$ belonging to sets of the 
form $ \Theta := \Theta_{x,t} := B_x(t) $, where $ x \in \Omega $ 
and $ t > 0 $ are such that $ \Theta \subseteq \Omega $.  For different choices of~$\Theta_{x,t} \ni z_i$, 
the lower bounds may be different, in which case the largest can be taken.  Broadly speaking, if the 
set~$\Theta_{x,t}$ is such that the metapopulation is sufficiently locally viable throughout it, in that
\[
     \eta_{\Theta_{x,t}} \Def \min_{z \in \Theta_{x,t}}q_1(z)
\]
is not close to zero, and if~$z_i$ is not too close to its boundary $\partial\Theta_{x,t}$, then the lower 
bound is reasonably close to~$q_1(z_i)$.

%Suppose that we have such a~$\Theta$.  Then,  for each $1 \le \b \le 1 + \tfrac12\eta_\Theta$, we can define  a positive quantity $\epsilon_{\Theta,\b}$ in terms of the parameters of the process, that is at least as big  as a positive multiple of~$\eta_\Theta^2$, provided that $r/t, L_q r, L_\s r$ and~$L_a r$ are all small enough;  the detailed requirements are in the statement of Lemma~\ref{lem:ULB} given in Section \ref{sec:LB}. With  this~$\epsilon_{\Theta,\b}$, we have the following result.

\begin{theorem}\label{thm:LB}
Suppose that\/ $\min_{z\in\Theta}q_{1}(z) =: \eta_{\Theta} > 0$, that $1 <  \beta < \beta^{\prime} < 1+ \eth/2 $, 
and that $ L_{f} \rmax /\emin > 1/2$. Assume that inequalities (\ref{lem:ULB:ineq1})--(\ref{lem:ULB:ineq3}), (\ref{lem:localLB:ineq1})  and~(\ref{lem:localLB:ineq2}) hold. Then, with 
\[
   m \Def \frac{\emin^2 \eth (\beta' - \beta)}{4 r \rmax^{2} L_{f} (C_1 \rmax + L_f)},
\]
we have
\begin{equation*}
    \mathbb{P}\left(p_{i,n}^{\ast} \geq p^-_{\Theta,\beta',m}(z_{i}) \mbox{ for all } z_{i} \in \Theta \right) 
    \ \geq\ 1- 2n \exp\left( -C_{2} \frac{(n-1)r^d}{A}\,
                   \frac{ C_4^{2}  \emin^2 \eth^{4} (\beta-1)^{2}}{\amax^2L_{f}^{2}} \right),
\end{equation*}
where~$C_{2}$ is as in Theorem~\ref{thm:UB} and~$C_4$ is a function of the elements of~$\QQ$, 
given in~\Ref{C7-def}.
\end{theorem}

Theorems \ref{thm:UB} and~\ref{thm:LB} can be combined in the following corollary.

\begin{corollary}\label{cor1}
Suppose that the conditions of Theorems \ref{thm:UB} and \ref{thm:LB} hold, with $ \alpha_{2} = 1- \eth < \a_1 < 1$
and with $\b,\b'$ and~$m$ as above. 
Define $ \Theta_{m} := \{z \in \Theta :  \| z - \partial \Theta\| \geq m^{-1} \}$, where 
$\Theta = B_x(t)$ and $t > r + m^{-1}$. Then
\begin{align*}
    \lefteqn{\mathbb{P}\left(\left| p_{i,n}^{\ast} - q_{1}(z_{i}) \right| \leq \a_1^{-1}(\beta' - \alpha_1) 
                               \mbox{ for all } z_{i} \in \Theta_{m} \right) } \\
   &\quad \geq\ 1- 4n \exp\left( - \frac{C_{2}(n-1)r^d \emin^{2}}{A\amax^{2} L_{f}^{2}} 
             \left( C_4^{2}   \eth^{4} (\beta-1)^{2} \wedge  \frac{(1-\alpha_{1})^{2}}{16} \right) \right) \\
           &\hskip1.4cm            - \frac n2 \exp\left( - c_d\smin (nr^{d}/A)\right).
\end{align*}
\end{corollary}

The inequalities (\ref{lem:ULB:ineq1})--(\ref{lem:ULB:ineq3}), (\ref{lem:localLB:ineq1}) 
and (\ref{lem:localLB:ineq2}) require that there are constants $C_q$ and~$C_{rt}$, functions only of the parameters 
in the set~$\QQ$, such that 
\eq\label{essential-conditions}
           L_{\rho} r + L_e r \Le C_q\eth(\eth \wedge  (\beta' - \beta));\quad 
      r/t \Le C_{rt}(\eth\wedge (\beta'-\beta)) .
\en 
As Corollary~\ref{cor1} shows, the approximation accuracy is better the closer $\b'$ and~$\a_1$ are to~$1$, whereas 
the probability that this accuracy is realized is reduced and the restrictions on $ L_{\rho} r + L_e r$ becomes more 
stringent as $ \b' $ and~$\a_1 $ become closer to~$1$. Furthermore,  if $\b'$ becomes closer to~$1$,  the subset 
of~$\Omega$ over which the approximation applies is smaller and the restriction on~$r/t$ becomes more stringent.

To illustrate the implications of these general results, we give a further consequence, expressed in the form of a  
limit theorem.  We think of a sequence of processes, indexed by~$n$, in which the parameters in~$\QQ$ are held 
constant, as are $L_\rho$ and~$L_e$, while the density of patches~$n/A_n$  increases.  Under such circumstances, 
it is reasonable to suppose that the colonization radius~$r_n$ decreases, since migrants can find other patches 
closer to home, but in such a way that $M_n := nr_n^d/A_n$ increases. 

The set~$\Omega_n$ is assumed to be somewhat more than $r_n$-smooth, in that we suppose that 
\[
   \Omega_n \Eq \bigcup_{i=1}^{X_n} B_{x(i,n)}(t(i,n)),
\] 
with $X_n < \infty$ and $t(i,n) \ge t_n$ for each~$i$,  where $r_n/t_n \le 1$ decreases with~$n$.  
The overlap in the union is also assumed not to be too great, in the sense that
\[
     \sum_{i=1}^{X_n} v_d \{t(i,n)\}^d \Le kA_n,
\]
for some~$k$ not depending on~$n$.  In consequence, for some~$k'$ not depending on~$n$,
\[
   N(\Omega,r_n/3) \Le k'A_n/r_n^d \Eq k'n/M_n \Eq O(n) \quad\mbox{as}\ n\to \infty.
\]

\begin{corollary}\label{cor2}
Under the above circumstances, suppose that $r_n \to 0$ as $n\to\infty$ and that~$t_n$ is bounded away from~$0$.  
Assume also that, for all~$n$, $ \eta_{\Omega_n}  \ge c_1 r_n^{\g_1} $ for some $\g_1 \in [0,1/2)$ and  $ c_{1} > 0$, 
and that
\eq\label{cor2-1}
       r_n^{2(1+\g_1)} \phi_n^2 M_n \ \ge\ c_2(\log n)^{1+\g_2},
\en
for some $c_2, \g_2 > 0$.  Then there exist constants $K_1,K_2 < \infty$ such that, 
for any sequence $\phi_n \to \infty$ such that $r_n^{1-2\g_1}\phi_n \to 0$,
\[
    \mathbb{P}\left(\left| p_{i,n}^{\ast} - q_{1}(z_{i,n}) \right| \leq K_1 r_n^{1-\g_1} \phi_n
                               \mbox{ for all } z_{i} \in \Omega'_n \right)  \ \to\ 1,
\]
where 
\[
    \Omega'_n \Def \{z \in \Omega\colon\, \|z-\partial \Theta^{(i,n)}\| \ge K_2/\phi_n \mbox{ for some } 
          1 \le i \le X_n\}.
\]
\end{corollary}

Thus, under such conditions, the error in the approximation converges to zero with $r_n^{1-\g_1}\phi_n$, and 
the proportion of~$\Omega_n$ for which the approximation does not hold converges to zero with~$1/\phi_n$ 
--- faster, if $t_n \to \infty$. When~$\eta_{\Omega_n} \to 0$ as $n\to\infty$,  the uniform precision is reduced.  
However, Corollary~\ref{cor2} could still be applied to any sequence of subsets $\widetilde\Omega_n \subset \Omega_n$ 
for which $\eta_{\widetilde\Omega_n}$ remains bounded away from zero, showing that the error is typically 
smaller where~$q_1(z)$ is larger.  Indeed, this illustrates the flexibility of our theorems.

Analogous results can also be proved in the limit in which the landscape becomes smoother, much
as in \citet{OC:06}, without necessarily requiring that $r_n \to 0$.  Assume instead that
\[
     s_n \Def \max\{r_n/t_n,\,(L_\rho^{(n)} + L_e^{(n)})r_n\} \ \to\ 0 \quad\mbox{as }n \to \infty.
\]
Then the statement of Corollary \ref{cor2} holds, with $s_n$ in place of~$r_n$.

\section{Discussion}

Both upper and lower bounds require the functions $e$ and~$\rho$ to be smooth.  One biologically relevant situation 
in which this need not be the case would be for terrestrial animals on islands, where, at the boundaries between 
sea and land, the patch density~$\s$ can be expected to change abruptly from a positive value to~$0$.  This is not 
a problem for the lower bound, since
the argument can be carried out within any subset of the region~$\Omega$ on which the functions $\rho$ and~$e$ are 
smooth.  For the upper bound, the argument given can be applied to any island~$\Omega'$ over which the functions 
$\rho$ and~$e$ are smooth, if~$\s(z) = 0$ at all points outside~$\Omega'$ to a distance of at least~$r$;  the proof 
of Theorem~\ref{thm:UB} indeed assumes that there is no contribution to the metapopulation coming from 
outside~$\Omega$.  On the intervening parts of~$\Omega$, in which the patch density~$\s$ is zero, there are no 
patches whose probability of occupancy is to be bounded.  If there are boundaries across which the values of the 
functions $\rho$ and~$e$ change abruptly, but not because~$\s$ jumps to zero, the upper bound argument would have 
to be modified in much the same spirit as that for the lower bound, but we have not attempted to do this.

Even when the region containing the habitat patches is connected, parts of the metapopulation can be rendered 
effectively disconnected by regions of low patch density, low colonisation rates and high extinction rates. In such 
circumstances, the deterministic metapopulation model may still possess a unique non-zero equilibrium. However,  
if the associated stochastic metapopulation model is initially in a state in which only some of the
viable regions are occupied, the remaining viable regions are likely to remain uncolonized for a very long time, 
so that the stochastic metapopulation model has quasi-equilibria that are not well approximated by the equilibria of the
deterministic system.

Our bounds on the equilibrium can be used to deduce bounds on the rate at which the metapopulation returns 
to equilibrium after a small displacement.  \citet{OH:02} refer to this as the `characteristic response time'.
 Near equilibrium, the continuous time system (\ref{Eq3}) can be appoximately expressed as
\[
\frac{d(p_{i}(t) - p_{i}^{\ast})}{dt}\ \approx\ \mathcal{J}(p^{\ast}) (p(t) - p^{\ast}),
\]
where
\[
  \mathcal{J}(p)\ =\ \left\{ 
    \begin{array}{ll}
             - \left( f (S_{i}(p^{\ast})) + e(z_i)\right), & i=j \\
            f^{\prime}(S_{i}(p^{\ast}))(1-p_{i}^{\ast}) \frac{A}{(n-1)} a(z_{j}) c(z_{i},z_{j};r), & i \neq j.
    \end{array}
  \right.
\]
% For appropriate choice parameters $ p^{-}_{\Theta,\beta,m} \leq p^{\ast} \leq p^{+}_{\alpha_{1},\alpha_{2}} $. 
Now suppose that we have functions $p^-$ and~$p^+$ such that $p^-(z_i) \le p_i^* \le p^+(z_i)$ for all
$1\le i\le n$.
Since, under Assumption~E, $ f $ is increasing and concave, 
$ \mathcal{J}(p^{+}) \leq \mathcal{J} ( p^{\ast}) \leq \mathcal{J}(p^{-}) $, 
where the inequality is interpreted elementwise. 
% As there exists a constant $ C_{5} $ such that 
Now, for any~$C_5$ chosen larger than $\max_{1\le i\le n}\{f(S_i(p^+))+ e(z_i)\}$, the matrix
$ C_{5} I + \mathcal{J}(p^{+}) $ is non-negative and primitive. 
By \citet[Theorem 1.1(e) of Chapter 1]{Seneta:81} it follows that 
$ \lambda (\mathcal{J}(p^{+})) \leq \lambda(\mathcal{J} ( p^{\ast})) 
\leq \lambda(\mathcal{J}(p^{-}) ) $, where $ \lambda(\cdot) $ denotes the leading eigenvalue. 
Thus we are able to bound the `characteristic response time' $ 1/\lambda(\mathcal{J}(p^{\ast})) $
of \citet{OH:02}, using $p^-$ and~$p^+$. 
 A similar argument can be made for the discrete time 
system (\ref{Eq2}), provided that $ f(S_{i}(p^{+})) + e(z_i) \leq 1 $ for all~$i$;
this ensures that the Jacobian is a non-negative matrix.  

If the function $ f $ is linear and $e(z) = \nu$ is constant in~$z$, $ q_{1}(z) $ 
as defined in Section~\ref{Sect3} is a concave function of $ \rho(z)$, provided that $ \rho(z) > \nu/L_f $. 
Jensen's inequality then implies that the equilibrium probability of patch occupancy, averaged over 
a region in which~$\rho(z)$ is uniformly above $\nu/L_f $, is smaller than the equilibrium probability 
of patch occupancy in a 
landscape with a constant colonisation rate equal to the spatial average. In this sense, spatial variability 
reduces the occupation level of the metapopulation when~$ f $ is linear. 
However, % $ q_{1}(z) $ is concave 
for strictly concave~$ f $ satisfying Assumption~E, 
$q_1(z)$ is not necessarily concave whenever $\rho(z) > \nu/L_f$.  
% In particular, if $f''(0) < 0$, $q_1(z)$ is not concave at points at which
% the colonisation rate is too close to $\nu/L_f$. % (see Lemma \ref{lem:concave}).
 %\adbr{Whatever we conclude about other functions~$f$}.

In the model that we discuss, randomness appears only through the positions of the patches in the smooth landscape. 
However, it would also be interesting to allow for the possibility that, although the landscape is smooth `on average', 
individual patches may have properties that differ from the average; for instance, the local extinction rates could 
be modelled as being random, with a mean that varies smoothly within the region.  It would also be interesting to 
allow the patch locations to be chosen as a sample from a point process with more dependence structure.

Another way in which additional randomness could be incorporated into the landscape is by allowing the landscape 
to change over time,
as in \cite{CO:08}. A common model for landscape dynamics is to allow habitat patches to change between 
`suitable' and `unsuitable' states, following a Markov chain (for example, \cite{KMVHL:00}). When a habitat patch 
becomes `unsuitable', any local population occupying that habitat patch becomes extinct, and the patch cannot be 
recolonised until it becomes `suitable' again. \citet{XFAS:06} incorporated this type of dynamics into the 
spatially realistic Levins model. Letting $ h(t,z_{i}) $ denote the probabilty that the habitat patch at 
$ z_{i} $ is `suitable', their system of equations becomes
\[
    \frac{dp_{i}(t)}{dt} \Eq f(S_{i}(p(t)))(h(t,z_{i})-p_{i}(t)) - \te(z_{i}) p_{i}(t),
\]
where~$\te(z_i)$ incorporates the rate of destruction of habitat patch~$i$, in addition to the rate of local extinction
$e(z_i)$ at patch~$i$.
Since $ h(t,z) $ converges to some  $ h(z) $ as $ t\rightarrow \infty$, the equilibrium probabilities for the spatially 
realistic Levins model with landscape dynamics are a fixed point of the function~$\tE_n$ given by
\[
\tE_{n}(p)_{i} \Def 
        \frac{h(z_{i}) f\left((A/(n-1)) \sum_{j\neq i} a(z_{j}) c(z_{i},z_{j};r) p_{j}\right)}
             { \te(z_{i}) + h(z_{i}) f\left((A/(n-1)) \sum_{j\neq i} a(z_{j}) c(z_{i},z_{j};r) p_{j}\right)}.
 \] 
It would thus be relatively straightforward to extend our analysis to bound the equilibria
of the deterministic model in \citet{XFAS:06}. In particular, if $ f $ is linear, 
then the equilibrium is approximated by $ 1  - \te(z_{i})/(h(z_{i})\rho(z_{i})) $.

\section{Appendix: Auxiliary results}
\setcounter{equation}{0}

Define the function $ F(\cdot;\tau,\nu)\colon\,[0,\infty) \rightarrow [0,1] $ by
\[
    F(x;\tau,\nu) \Def \frac{f(\tau x)}{\nu + f(\tau x)}.
\]
This function has a fixed point at $ 0 $ and, if $ f^{\prime}(0) \tau > \nu$, then it also has a non-zero fixed point. The function  $ q_{\alpha}(z) $ is the largest fixed point of $ F(\cdot;\rho(z),\alpha e(z)) $, for fixed $ \alpha > 0 $. 

\begin{lemma} \label{lem:Fbound2}
Suppose that Assumption E holds. Let $ q $ denote the largest fixed point of $ F(\cdot;\tau,\nu) $. 
Then $ q \leq x $ if   $ (1-x) f(\tau x) \leq \nu x $ and $ q \geq x $ if  $ (1-x) f(\tau x) \geq \nu x $. 
\end{lemma}

\begin{proof}
Since $ f $ is concave, increasing and not identically zero, by Assumption~E, $ F(\cdot;\tau,\nu) $ is concave, and strictly concave at~$0$. Hence $g(x) := F(x;\t,\n) - x$ is concave, strictly concave at zero, and has $g(0)=0$ and $g(\infty) = -\infty$. 
If $g'(0) = F'(0;\t,\n) - 1 \le 0$, there is thus no other solution to $g(x)=0$. If $g'(0) > 0$, there is exactly one other solution~$q$, and $g(x) > 0$ for $0 < x < q$, and $g(x) < 0$ for $x > q$.  Thus, $0 < x < q$ if and only if $g(x) > 0$, and so
\[
    F(x;\t,\n) \Eq \frac{f(\tau x)}{\nu + f(\tau x)} \ >\  x,
\]
implying that $(1-x)f(\tau x) > \nu x$; similarly, $q < x$ if and only if $g(x) < 0$ and $(1-x)f(\tau x) < \nu x$.

\end{proof}

\begin{lemma} \label{lem:Lips}
Suppose that Assumptions C, D and E  hold. Then, for all $\a \ge 1/2$, $ q_{\alpha} $ is Lipschitz continuous on $ \{z \in \Omega : q_{\alpha}(z) > 0\} $, with Lipschitz constant at most~$L_q$, as defined in~\Ref{lem:dif:eq4}.
\end{lemma}

\begin{proof} 
We write $ F_\a(q,z) := F(q;\rho(z),\alpha e(z)) $, where $F_\a\colon\,[0,1] \times \re^d \to \re_+$. For functions $g\colon\,[0,1] \times \re^d \to \re_+$, we denote by~$D_q g$ the partial derivative of~$g$ with respect to its first argument, and by~$D_j g$  the partial derivative in the direction of the $j$-th coordinate axis in~$\re^d$, $1\le j\le d$. By the implicit function theorem, $ q_{\alpha}(z) $ is continuously differentiable in an open neighbourhood of $ z $, with 
\begin{equation}
    (D_j q_\a)(z) \Eq - \left[(D_q F_\a)(q_{\alpha}(z),z) -1\right]^{-1} 
                (D_j F_\a)(q_{\alpha}(z),z),\quad 1\le j\le d, \label{lem:dif:eq0}
\end{equation}
provided that
\begin{equation}
  (D_q F_\a)(q_{\alpha}(z),z)\ \neq\ 1. \label{lem:dif:eq1}
\end{equation}
For any $ z \in \Omega$ and $q \in [0,1]$, 
\begin{align*}
   (D_q F_\a)(q,z) & \Eq \frac{\alpha e(z) \rho(z) f^{\prime}(q \rho(z))}{\left(\alpha e(z) + f(q\rho(z))\right)^{2}}.
\end{align*}
As $ q_{\alpha}(z) = F(q_{\alpha}(z);\alpha e(z),\rho(z)) $, 
\begin{align*}
   (D_q F_\a)(q_{\alpha}(z),z) & \Eq 
    \frac{(1-q_{\alpha}(z)) \rho(z) f^{\prime}(q_{\alpha}(z) \rho(z))}{\left(\alpha e(z) + f(q_{\alpha}(z)\rho(z))\right)}  \\
   & \Eq  \frac{(1-q_{\alpha}(z)) q_{\alpha}(z)\rho(z) f^{\prime}(q_{\alpha}(z) \rho(z))}{ f(q_{\alpha}(z)\rho(z))}.
\end{align*}
By the mean value theorem, there exists a $ \tilde{q} \in (0,q_{\alpha}(z)) $ such that 
\[
      f^{\prime}(\tilde{q} \rho(z)) \Eq \frac{f(q_{\alpha}(z)\rho(z))}{q_{\alpha}(z) \rho(z)}.
\]
As $ f $ is concave, $ f^{\prime}(\tilde{q} \rho(z)) \geq f^{\prime}(q_{\alpha}(z) \rho(z)) $. Therefore,
\begin{equation}
   (D_q F_\a)(q_{\alpha}(z),z)  \Le 1-q_{\alpha}(z),  \label{lem:dif:eq2}
\end{equation}
and (\ref{lem:dif:eq1}) holds for any $ z \in \Omega $ such that $ q_{\alpha}(z) > 0 $. Differentiating~$ F_\a $ in direction~$j$ yields
\begin{align*}
   (D_j F_\a)(q,z) & \Eq \frac{\alpha e(z) q (D_j \rho)(z) f^{\prime}(q\rho(z)) - \alpha (D_j e)(z) f(q\rho(z))}
           {\left(\alpha e(z) + f(q\rho(z)) \right)^{2}}.
\end{align*}
Evaluating this derivative at $ (q_{\alpha}(z),z) $ gives
\begin{align}
    (D_j F_\a)(q_{\alpha}(z),z) & \Eq \frac{\alpha e(z) q_{\alpha}(z) (D_j \rho)(z) f^{\prime}(q_{\alpha}(z)\rho(z))
         - \alpha(D_j e)(z) f(q_{\alpha}(z) \rho(z))}{\left(\alpha e(z) + f(q_{\alpha}(z)\rho(z)) \right)^{2}} \nonumber\\
    & \Eq q_{\alpha}(z)\frac{(1-q_{\alpha}(z))  (D_j \rho)(z) f^{\prime}(q_{\alpha}(z)\rho(z)) 
         - \alpha (D_j e)(z) }{\left(\alpha e(z) + f(q_{\alpha}(z)\rho(z)) \right)}. \label{lem:dif:eq3}
\end{align}
Combining equations (\ref{lem:dif:eq0}) and (\ref{lem:dif:eq3}) with the bound (\ref{lem:dif:eq2}) yields 
\begin{align*}
     \left|(D_j q_\a)(z) \right| & \Le \frac{1}{e(z)} 
        \left( \frac{L_f}{\alpha} \left|(D_j \rho)(z) \right| + \left| (D_j e)(z) \right| \right)
            \Le \frac1{\emin}(\a^{-1}L_fL_\rho + L_e).
\end{align*}
Therefore, for any $\a \ge 1/2$, $ q_{\alpha} $ is Lipschitz on $ \{z\in\Omega: q_{\alpha} > 0\} $ with the 
Lipschitz constant given in~(\ref{lem:dif:eq4}). 
\end{proof}

\begin{lemma} \label{lem:Fbound}
Suppose that Assumption~E holds and that $ q_{1}(z) \geq \eta > 0$. Then, for any $ \b \in (1, (1-\eta)^{-1}) $, 
$ q_{\b}(z) \geq  \b\eta + 1-\b $, and, for any $ \alpha \in (0,1) $, $ q_{\alpha}(z) \geq \alpha \eta $.
\end{lemma}

\begin{proof}
For any $ \b \in (1, (1-\eta)^{-1}) $, it follows that $0 < \b \eta + 1- \b < \eta $ and that, by Assumption~E,
\[
    f\left( \rho(z)(\b \eta + 1 - \b)  \right)\ \geq\ \left( \b + \frac{1-\b}{\eta}\right) f(\rho(z) \eta). 
\]
As $ q_{1}(z) \geq \eta $, we can apply Lemma \ref{lem:Fbound2} to give $ f(\rho(z) \eta) \ge e(z)\eta/(1-\eta)$, and hence
\eqs
    f\left( \rho(z)(\b \eta + 1 - \b)  \right) 
          &\ge& \left( \b + \frac{1-\b}{\eta}\right) \frac{e(z)\eta}{1-\eta} \\
        &=& (\b\eta + 1 - \b) \frac{e(z)}{1-\eta}\
          \Eq    \frac{\b e(z)(\b \eta + 1 - \b)}{1- (\b \eta + 1 - \b)}.
\ens
Applying Lemma~\ref{lem:Fbound2} again, we see that $ q_{\b}(z) \geq \b\eta +1 -\b$. 

For $ \alpha \in (0,1)$ we follow similar reasoning to show that 
$$
   f(\rho(z) \alpha \eta)\ \geq\ \alpha f(\rho(z) \eta)\ \geq\ \alpha e(z) \eta/(1-\eta)\ 
       \geq\ \alpha e(z) \alpha \eta/ (1-\alpha \eta), 
$$ 
and applying Lemma~\ref{lem:Fbound2} we see that $ q_{\alpha}(z) \geq \alpha\eta $.
\end{proof}

In the following we let $ \sigma_{n\setminus i} := \frac{A}{(n-1)}\sum_{j\neq i} \d_{z_j}$, which is 
$ A $ times the empirical measure of patches excluding patch $ i $.

\begin{lemma} \label{lem:emp}

Suppose that Assumptions A, B and D hold. Then, for any $ h\colon\, \Omega \rightarrow [0,H] $, 
 $0 <  \ddd \le H\cmax\smax v_d$  and  $ z \in \Omega $,
\[
  \mathbb{P} \left(  \left|\int c(z,y;r) h(y) [\sigma_{n\setminus i}(dy) - \sigma(y)\,dy] \right|
    \ \geq\ \ddd\right)  \leq  2 \exp\left( - C_2 ((n-1) r^d/A) (\ddd/H)^{2}\right),
\]
where 
\[
    C_2 \Def \{3\{\cmax\}^2 \smax v_d\}^{-1}.
\]
\end{lemma}

\begin{proof}
Note first that, for patches distributed independently with density $A^{-1}\s(\cdot)$, we have 
\[
    \ex\left\{\int c(z,y;r) h(y) \sigma_{n\setminus i}(dy) \right\}
         \Eq \int c(z,y;r) h(y) \s(y)\,dy.
\]
The left hand side of this expression is a sum of \iid\ random variables, each bounded by $H\cmax A/((n-1)r^d)$,
and each with variance at most $\{H\cmax A/((n-1)r^d)\}^2 \smax v_d r^d/A$, where, as before, $v_d$ denotes the volume
of the unit ball in~$\re^d$. Hence, applying \citet[Theorem 2.7]{McDiarmid:98}, it follows that, for
any $\ddd > 0$,
\eqs
   \lefteqn{\pr\left( \left|\int c(z,y;r) h(y) \sigma_{n\setminus i}(dy) - \int c(z,y;r) h(y) \s(y)\,dy \right| > \ddd\right)}\\
      &&\Le 2\exp\left( - \frac{\ddd^2}{2((n-1)\{H\cmax A/((n-1)r^d)\}^2 \smax v_d A^{-1}r^d
                      + \{H\cmax A/((n-1)r^d)\}\ddd/3)} \right) \\
      &&\Le 2\exp\left( - \frac{\ddd^2}{3(A/(n-1)r^d)\{H\cmax\}^2 \smax v_d } \right)\\
      &&\Eq 2\exp\left( - C_2((n-1)r^d/A)(\ddd/H)^2 \right)
\ens
if $\ddd/H \le \cmax\smax v_d$.
\end{proof}

\begin{lemma} \label{lem:T}
Suppose Assumptions A, B and D  hold. Let $ N(\Omega,r) $ be the number of balls of radius $ r $ required 
to cover $ \Omega $. If $ n > 2 N(\Omega,r/3) $, then $ T $ is primitive with probability at least 
\eq \label{lem:T:eq1}
1 - N(\Omega,r/3) \exp\left( - n \min_{z\in\Omega} A^{-1}\int_{\Omega} \mathbb{I}(\|y-z\|\leq r/3) \sigma(y) dy\right)
\en
\end{lemma}

\begin{proof}
Let $ \tilde{T} $ be the incidence matrix of $ T $, that is $ \tilde{T}_{ij} = 1 $ if $ T_{ij} > 0 $ and 
$ T_{ij} = 0 $ otherwise. The matrix $ T $ is primitive if $ \tilde{T} $ is both irreducible and 
acyclic \cite[Theorem 1.4 of Chapter 1]{Seneta:81}. By Assumptions $ B $ and $ D $, $ \tilde{T} $ is 
symmetric and $ \tilde{T}_{ii} = 0 $. Define the graph $ \mathcal{G} = (V,E) $ where $ V:= \{z_{1},\ldots,z_{n}\} $ 
and $ (z_{i},z_{j}) \in E $ if and only if $ \|z_{i} - z_{j} \| \leq r $. The matrix $ \tilde{T} $ is the 
adjacency matrix of $ \mathcal{G} $ and is irreducible if $ \mathcal{G} $ is connected.  
Let $ N:= N(\Omega,r/3) $ and $ y_{1},\ldots,y_{N} \in \Omega $ such that $ \Omega \subset \cup_{i}^{N} B(y_{i},r/3) $, 
where $ B(y,r) $ is a closed ball of radius $ r $ centered at $ y $. Define the graph 
$ \hat{\mathcal{G}} = (\hat{V},\hat{E}) $ where $ \hat{V} = \{y_{1},\ldots,y_{N}\} $ and $ (y_{i},y_{j}) \in \hat{E} $ 
if and only if $ \|y_{i} - y_{j} \| \leq r/3 $. Since $ \Omega $ is connected, the graph $ \hat{\mathcal{G}} $ 
is also connected. Suppose that each ball $ B(y_{i}, r/3) $ contains at least one element of $ V $. 
For any $ z_{i} $ and $ z_{j} $, there exists a path $ \{y_{k_0},y_{k_1},\ldots, y_{k_{m+1}}\} $ in 
$ \hat{\mathcal{G}} $ such that $ z_{i} \in B(y_{k_0},r/3) $ and $ z_{j} \in B(y_{k_{m+1}},r/3) $. 
Taking any $ z_{k_\ell} \in B(y_{k_\ell},r/3) $, we have constructed a path $ \{z_{i},z_{k_1},\ldots,z_{k_m},z_{j}\} $ 
in $ \mathcal{G} $, since 
\[
\| z_{k_\ell} - z_{k_{\ell+1}} \| \leq  \| z_{k_\ell} - y_{k_\ell} \| + \| y_{k_\ell} - y_{k_{\ell+1}} \| + \| y_{k_{\ell+1}} - z_{k_{\ell+1}} \| \leq r.
 \] 
Thus $ \mathcal{G} $ is connected and $ \tilde{T} $ is irreducible if  each ball $ B(y_{i}, r/3) $ contains at 
least one element of $ V $. This occurs with probability at least that given in (\ref{lem:T:eq1}).

To show that $ \tilde{T} $ is acyclic, it is sufficient to show that $ \tilde{T}^{2}_{ii} > 0 $ 
and $ \tilde{T}^{3}_{ii} > 0 $ for some~$i $, since $ \tilde{T} $ is irreducible \cite[Lemma 1.2 of Chapter 1]{Seneta:81}. 
This is true if there are three elements of $ V$  that are within distance~$r$ of each other. Since $ n > 2N(\Omega,r/3) $, 
there is at least one $ B(y_{i},r/3) $ which contains at least three elements of $ V $, and these are within distance
$2r/3$ of each other, completing the proof.
\end{proof}

\section{Appendix: Proof of the upper bound}
\setcounter{equation}{0}

In this section, we prove Theorem~\ref{thm:UB}. Suppose
\begin{equation}
    f\left(\int a(y) c(z_{i},y;r) p_{\aot}^{+}(y) \sigma_{n\setminus i}(dy)\right) 
         \ \leq\ \frac{e(z_{i}) p_{\aot}^{+}(z_{i})}{1-p_{\aot}^{+}(z_{i})} \label{lem:ub:eq0}
\end{equation}
for all $ i = 1,\ldots,n $. Then $ E_{n} (p_{\aot}^{+}(z))_{i} \leq p_{\aot}^{+}(z_{i}) $ for all $ i=1,\ldots,n$. As $ E_{n} $ is monotone, the sequence of iterates of $ E_{n} $ starting from $ p_{\aot}^{+}(z_{i}),\ i=1,\ldots,n$ is decreasing. If $ T $ is primitive, then the cone limit set trichotomy \cite[Theorem 6.3]{HS:05} holds and each sequence of iterates starting from a non-zero inital value must converge to $ p^{\ast} $. Hence, $ p_{\aot}^{+}(z_{i}),\ i=1,\ldots,n$ is an upper bound on $ p^{\ast}$.  The matrix $ T $ is primitive with high probability by Lemma \ref{lem:T}.  It remains to show that for some $1/2 < \a_2 \le \a_1 < 1$ inequality (\ref{lem:ub:eq0}) holds. 

Since $ c(z,y;r) = 0 $ for all $ y $ such that $ \| y-z\| > r $, and since $p_{\aot}^{+} $ is Lipschitz with constant~$L_q$,
as given in~\Ref{lem:dif:eq4}, we have
\begin{align}
    f\left(\int a(y)c(z,y;r) p_{\aot}^{+}(y) \sigma_{n\setminus i}(dy)\right) 
        & \Le f\left(\int a(y)c(z,y;r) \left[p_{\aot}^{+}(z) + L_q r\right] \sigma_{n\setminus i}(dy)\right) \nonumber\\
        & \Le f\left(\rho_{n\setminus i}(z) p_{\aot}^{+}(z)\right) + L_{f} L_q \rho_{n\setminus i}(z) r, \label{lem:ub:eq1}
\end{align}
for $\a_{2} > 1/2$, where $ \rho_{n\setminus i}(z) := \int a(y) c(z,y;r) \sigma_{n\setminus i}(dy) $. 

{\bf U1.} For all $ z $ such that $ q_{\alpha_{1}}(z) < 1 - \alpha_{2} $,
\begin{align}
   \lefteqn{ \frac{e(z) p_{\aot}^{+}(z)}{1-p_{\aot}^{+}(z)} 
             - f\left(\int a(y)c(z,y;r) p_{\aot}^{+}(y) \sigma_{n\setminus i}(dy)\right)} \nonumber \\
  & \geq\ 
      \frac{e(z) (1-\alpha_{2})}{\alpha_{2}} - f(\rho_{n}(z) (1-\alpha_{2})) -  L_{f} L_q \rho_{n\setminus i}(z) r \nonumber \\
  & \geq\ \frac{e(z) (1-\alpha_{2})}{\alpha_{2}} - f(\rho(z) (1-\alpha_{2}))  
          - L_{f} ((1-\alpha_{2}) + L_q r) \left|\rho(z) - \rho_{n\setminus i}(z) \right|-  L_{f} L_q \rho(z) r. \label{lem:ub:eq2}
\end{align}
From Lemma \ref{lem:Fbound2}, if $ q_{\alpha_{1}}(z) \leq (1-\alpha_{2}) $, then $ f(\rho(z)(1-\alpha_{2})) \leq \alpha_{1}(1-\alpha_{2}) e(z)/\alpha_{2} $. Combining this bound with inequality (\ref{lem:ub:eq2}) gives
\begin{align}
  \lefteqn{ \frac{e(z) p_{\aot}^{+}(z)}{1-p_{\aot}^{+}(z)} 
           - f\left(\int a(y)c(z,y;r) p_{\aot}^{+}(y) \sigma_{n\setminus i}(dy)\right)} \nonumber\\
  & \geq\ \frac{e(z)(1-\alpha_{1})(1-\alpha_{2})}{\alpha_{2}} - L_{f} ((1-\alpha_{2}) + L_qr) \left|\rho(z) - \rho_{n\setminus i}(z) \right| -  L_{f} L_q \rho(z) r \nonumber\\
  & \geq\ p_{\aot}^{+}(z) \left(\frac{(1-\alpha_{1}) e(z)}{\alpha_{2}} -  \frac{L_{f} L_q r \rho(z) }{ p_{\aot}^{+}(z)}   
     -  L_{f} \left(1 + \frac{L_q r}{ p_{\aot}^{+}(z)} \right) \left|\rho(z) - \rho_{n\setminus i}(z) \right| \right), 
               \label{lem:ub:eq4a}
\end{align}
where the last inequality follows as $ q_{\alpha_{1}}(z) < 1-\alpha_{2} $ implies $ p_{\aot}^{+}(z) = 1-\alpha_{2} $. 

{\bf U2.} We now consider the case where $ q_{\alpha_{1}} (z) \geq 1- \alpha_{2} $. Using the fact that $ q_{\alpha_{1}}(z) $ is a fixed point of $ F(\cdot;\rho(z),\alpha_{1} e(z)) $ and inequality (\ref{lem:ub:eq1}), it follows that
\begin{align}
   \lefteqn{ \frac{e(z) p_{\aot}^{+}(z)}{1-p_{\aot}^{+}(z)} 
           - f\left(\int a(y) c(z,y;r) p_{\aot}^{+}(y) \sigma_{n\setminus i}(dy)\right)} \nonumber \\
   & \geq\ \frac{f(\rho(z) q_{\alpha_{1}}(z))}{\alpha_{1}} - f(\rho_{n\setminus i}(z) q_{\alpha}(z)) - L_{f}L_q r \rho_{n\setminus i}(z) \nonumber \\
   & \geq\ \frac{(1-\alpha_{1})}{\alpha_{1}} f(\rho(z) q_{\alpha_{1}}(z)) - L_{f} L_q r \rho(z)  
             - L_{f} (q_{\alpha}(z) + L_q r) \left| \rho(z)  - \rho_{n\setminus i}(z)\right|.  \label{lem:ub:eq3}
\end{align}
Since $ f(\rho(z) q_{\alpha_{1}}(z)) = \alpha_{1} e(z) q_{\alpha_{1}}(z)/(1-q_{\alpha_{1}}(z)) \geq \alpha_{1} e(z) q_{\alpha_{1}}(z) $ and 
$ q_{\alpha_{1}}(z) \geq 1-\alpha_{2}$ here, inequality (\ref{lem:ub:eq3}) becomes
\begin{align}
   \lefteqn{ \frac{e(z) p_{\aot}^{+}(z)}{1-p_{\aot}^{+}(z)} 
       - f\left(\int a(y) c(z,y;r) p_{\aot}^{+}(y) \sigma_{n\setminus i}(dy)\right)} \nonumber \\
   & \geq\ q_{\alpha_{1}}(z) \left((1-\alpha_{1}) e(z)  - \frac{L_{f} L_q r \rho(z) }{q_{\alpha_{1}}(z)} 
     - L_{f} \left(1 + \frac{L_q r}{q_{\alpha_{1}}(z)}\right) \left| \rho(z)  - \rho_{n\setminus i}(z)\right|  \right) \nonumber \\
   & \geq\ p^{+}_{\alpha_{1},\alpha_{2}}(z) \left((1-\alpha_{1}) e(z)  - \frac{L_{f} L_q  r \rho(z) }{p_{\aot}^{+}(z)} 
        - L_{f} \left(1 + \frac{L_q r}{p_{\aot}^{+}(z)}\right) \left| \rho(z)  - \rho_{n\setminus i}(z)\right|  \right).  
           \label{lem:ub:eq4b}
\end{align}

{\bf U3.} 
Combining inequalities (\ref{lem:ub:eq4a}) and (\ref{lem:ub:eq4b}), we see that inequality (\ref{lem:ub:eq0}) holds if 
\[
(1-\alpha_{1}) e(z_{i})  - \frac{L_{f} L_q r \rho(z_{i})}{p_{\aot}^{+}(z_{i})} - 
    L_{f} \left(1 + \frac{L_q r}{p_{\aot}^{+}(z_{i})}\right) \left| \rho(z_{i})  - \rho_{n\setminus i}(z_{i})\right|   \geq 0
\]
which is equivalent to
\[
\frac{(1-\alpha_{1}) e(z_{i}) p_{\aot}^{+}(z_{i})  - L_{f} L_q r \rho(z_{i})}
        {L_{f}\left(p_{\aot}^{+}(z_{i}) + L_qr\right)}  
          \ \geq\ \left| \rho(z_{i})  - \rho_{n\setminus i}(z_{i})\right|.
\]
By Lemma \ref{lem:Fbound}, $ p_{\aot}^{+}(z)  \geq \bigl(\alpha_{1} \eta_{\Omega}  \vee (1-\alpha_{2})\bigr) $, and from inequality~(\ref{UB:ineq1}) together with $L_f\rmax/\emin > 1/2$, we see that inequality~(\ref{lem:ub:eq0}) is satisfied if 
\[
     \frac{(1-\alpha_{1}) \emin}{4L_{f}}\ \geq\ \left| \rho(z_{i})  - \rho_{n\setminus i}(z_{i})\right|.
\]
Applying Lemma~\ref{lem:emp} yields the bound
\begin{align*}
    \lefteqn{\mathbb{P} \left( \max_{i}  \left|\int a(y)c(z_{i},y;r) [\sigma_{n\backslash i}(dy) - \sigma(y)\,dy]  \right| 
     \le  \frac{(1-\alpha_{1})\emin}{4L_{f}}\right)}\\
   &  \geq\ 1 - n \sup_{z\in \Omega} \mathbb{P} \left( \left|\int a(y)c(z,y;r) [\sigma_{n\backslash i}(dy) - \sigma(y)\,dy] 
             \right|  > \frac{(1-\alpha_{1})\emin}{4L_{f}} \right) \\
   & \geq\ 1 - 2n \exp\left( -C_2\{(n-1)r^d/A\}\frac{\emin^{2} (1-\alpha_{1})^{2}}{16 \amax^2 L_{f}^{2}}  \right),
\end{align*}
if $\frac{(1-\alpha_{1})\emin}{4L_{f}} \le \rmax$, which is the case if $L_f\rmax/\emin > 1/2$.

The situation in which $L_f\rmax/\emin \le 1/2$ is one in which the metapopulation is nowhere viable, so the conclusion is not surprising. We begin by noting that $q_\a(z) = 0$ for all $\a > 1/2$ if  $L_f\rmax/\emin \le 1/2$,  so that $p^{+}_{\alpha_{1},\alpha_{2}}(z) = (1-\a_2)$ for all $z\in\Omega$. Lemma~\ref{lem:emp} with $t = \rmax/2$ then shows that
\[
   \left| \rho(z_{i}) - \rho_{n\setminus i}(z_{i})\right| \Le  \frac{1}{2} \rmax,
\]
on an event of probability at least
\[
   1 - 2 \exp\left\{ -C_2 \frac{(n-1) r^d}{A} \Bigl(\frac{ \rmax}{2 \amax}\Bigr)^2    \right\}.
\]
Hence, on this event, we have
\eqs
f\left((1-\a_2)\int a(y)c(z_{i},y;r)  \sigma_{n\setminus i}(dy)\right) 
           \Le \frac32(1-\a_2)L_f \rmax   \Le (1-\a_2)\emin \Le \frac{e(z_i)(1-\a_2)}{\a_2}.
\ens
This establishes~\Ref{lem:ub:eq0}, on an event with probability as given in Theorem~\ref{thm:UB}, for any choice of $1/2 < \a_2 < 1$, since $p^{+}_{\alpha_{1},\alpha_{2}}(z) = (1-\a_2)$ for all $z\in\Omega$.  This completes the proof of Theorem~\ref{thm:UB}.

\section{Appendix: Proof of the lower bound} \label{sec:LB}
\setcounter{equation}{0}

To find a good lower bound on $ p^{\ast}_{n} $, we introduce a modification of $ E_{n} $. For any 
$ \Theta \subseteq \Omega $ and $ \beta >1 $ define the operator $ E_{n,\Theta,\beta}: [0,1]^{n} \rightarrow [0,1]^{n} $ 
by
\[
  E_{n,\Theta,\beta}(p)_{i} \Def  
        \frac{f\left((A/(n-1)) \sum_{j\neq i} a(z_{j}) c(z_{i},z_{j};r) \mathbb{I}(z_{j} \in \Theta) p_{j}\right)}
    { \beta e(z_{i}) + f\left((A/(n-1)) \sum_{j\neq i} a(z_{j}) c(z_{i},z_{j};r) \mathbb{I}(z_{j} \in \Theta) p_{j}\right)}.
\]
Denote the largest fixed point of $ E_{n,\Theta,\beta} $ by $ p_{n,\Theta,\beta}^{\ast} $. Since $ f $ is an 
increasing function, for any $ \Theta \subseteq \Omega $ and any $ \beta >1 $, 
$ E _{n,\Theta,\beta} (p) \leq E_{n,\Theta,1}(p) \leq E_{n}(p)$ for all  $ p \in [0,1]^{n} $,  which 
implies that $ p_{n,\Theta,\beta}^{\ast} \leq p^{\ast}_{n,\Theta,1} \leq p_{n}^{\ast} $. Thus a lower bound 
on $ p_{n,\Theta,\beta}^{\ast} $ yields a lower bound on $ p_{n}^{\ast} $. To construct a lower bound on 
$ p_{n,\Theta,\beta}^{\ast} $, we examine the limiting form of $ E_{n,\Theta,\beta} $ as $ n \rightarrow \infty $. 
Let $ C^{+}(\Theta) $ be the set of non-negative functions on $ \Theta $ and define $ E_{\Theta,\beta} \colon\, C^{+}(\Theta) \rightarrow C^{+}(\Theta) $ by 
\[
   E_{\Theta,\beta}(p) \Def  
       \frac{f\left(\int a(y) c(z,y;r) \mathbb{I}(y \in \Theta) p(y) \sigma(dy) \right)}
      { \beta e(z) + f\left(\int a(y) c(z,y;r) \mathbb{I}(y \in \Theta) p(y) \sigma(dy) \right)}.
\]
Let  $ p_{\Theta,\beta}^{\ast} $ denote the largest fixed point of $ E_{\Theta,\beta} $. Our aim now is to find 
a $ \beta > 1 $ such that with high probability
\begin{equation}
    p_{i,n}^{\ast}\ \geq\ p^{\ast}_{\Theta,\beta}(z_{i}), \label{LB:eq1}
\end{equation}
for all $ z_{i} \in \Theta $. 

\begin{lemma} \label{lem:lb1}
Suppose that Assumptions A, B, D and E hold. Suppose also that, for a given $ \Theta \subseteq \Omega $ and $ \beta > 1 $, 
there exists an $ \epsilon_{\Theta,\beta} > 0 $ such that $ p_{\Theta,\beta}^{\ast} (z) \geq \epsilon_{\Theta,\beta} $ 
for all $ z \in \Theta $. Assume that 
\begin{align}
  \emin \epsilon_{\Theta,\b} (\beta-1) &\Le L_f\rmax. \label{LB:ineq3}
\end{align}
Then
\begin{equation}
    \mathbb{P}\left(p_{i,n}^{\ast} \geq p^{\ast}_{\Theta,\beta}(z_{i}) \mbox{ for all } z_{i} \in \Theta \right) 
    \ \geq\ 1- 2n \exp\left( -C_2 \frac{(n-1)r^d}{A}
                \frac{ \emin^2 \epsilon_{\Theta,\beta}^{2} (\beta-1)^{2}}{4\amax^2L_{f}^{2}} \right).
      \label{LB:eq2}
\end{equation}
\end{lemma}

\begin{proof}
Suppose that
\eq
    \frac{ e(z_{i}) p_{\Theta,\beta}^{\ast}(z_{i})}{1 - p_{\Theta,\beta}^{\ast}(z_{i})} \Le
        f\left( \int a(y) c(z_{i},y;r) \mathbb{I}(y \in \Theta) p_{\Theta,\beta}^{\ast}(y) \sigma_{n\setminus i} (dy) \right)  \label{lem:lb1:eq0}
\en
for all $ z_{i} \in \Theta $.  Then $ E_{n,\Theta,1} $ maps the set $ \{p: p_{\Theta,\beta}^{\ast}(z_{i}) \leq p_{i} \leq 1\} $ into itself as the map is monotone. Applying the Brouwer fixed point theorem, we see  $ p_{\Theta,\beta}^{\ast}(z_{i}) \leq p_{n,\Theta,1,i} $ for all $ z_{i} \in \Theta $. Since $  p^{\ast}_{n,\Theta,1} \leq p^{\ast}_{n} $, it remains to verify inequality (\ref{lem:lb1:eq0}) holds.

Now
\begin{align*}
  \lefteqn{ f\left( \int a(y) c(z_{i},y;r) \mathbb{I}(y \in \Theta) p_{\Theta,\beta}^{\ast}(y) \sigma_{n\setminus i} (dy) \right)}  \\
  & =\  f\biggl(\int a(y) c(z_{i},y;r) \mathbb{I}(y \in \Theta) p_{\Theta,\beta}^{\ast}(y) \sigma(y)\,dy \\
   &\qquad\qquad\mbox{} 
        +  \int a(y) c(z_{i},y;r)  \mathbb{I}(y \in \Theta) p_{\Theta,\beta}^{\ast} (y) [\sigma_{n\setminus i}(dy) - \sigma(y)\,dy]
                                   \biggr) \\
  & \geq\  f\biggl(\int a(y) c(z_{i},y;r) \mathbb{I}(y \in \Theta) p_{\Theta,\beta}^{\ast}(y) \sigma(y)\,dy \biggr) \\
   &\qquad\qquad\mbox{} 
     - L_{f} \left|\int a(y) c(z_{i},y;r)  \mathbb{I}(y \in \Theta) 
                          p_{\Theta,\beta}^{\ast}(y) [\sigma_{n\setminus i}(dy) - \sigma(y)\,dy]  \right| \\
  & \geq\ \frac{\beta e(z_{i}) p_{\Theta,\beta}^{\ast}(z_{i})}{1-p_{\Theta,\beta}^{\ast}(z_{i})} 
     -  L_{f} \left|\int a(y) c(z_{i},y;r)  \mathbb{I}(y \in \Theta) 
                p_{\Theta,\beta}^{\ast}(y)[\sigma_{n\setminus i}(dy) - \sigma(y)\,dy]  \right| .
\end{align*}
Therefore,
\begin{align*}
    \lefteqn{f\left( \int a(y) c(z_{i},y;r) \mathbb{I}(y \in \Theta) p_{\Theta,\beta}^{\ast}(y) \sigma_{n\setminus i} (dy) \right) - \frac{ e(z_{i}) p_{\Theta}^{\ast}(z_{i})}{1 - p_{\Theta,\beta}^{\ast}(z_{i})}} \\
   &  \geq\ \frac{(\beta-1) e(z_{i}) p_{\Theta,\beta}^{\ast}(z_{i})}{1-p_{\Theta,\beta}^{\ast}(z_{i})} 
      -  L_{f} \left|\int a(y) c(z_{i},y;r) \mathbb{I}(y \in \Theta) 
                      p_{\Theta,\beta}^{\ast}(y)[\sigma_{n\setminus i}(dy) - \sigma(y)\,dy]  \right| .
\end{align*}
As  $ p_{\Theta}^{\ast} (z) \geq \epsilon_{\Theta,\beta} $ for all $ z \in \Theta $, inequality (\ref{LB:eq1}) will hold  if 
\begin{equation}
    \frac{(\beta-1)\emin\epsilon_{\Theta,\beta}}{L_{f}} -  
     \left|\int a(y) c(z_{i},y;r)  \mathbb{I}(y \in \Theta) 
        p_{\Theta,\beta}^{\ast} (y) [\sigma_{n\setminus i}(dy) - \sigma(y)\,dy]  \right|\ \geq\ 0, \label{lem:lb1:eq1}
\end{equation}
for all $ z_{i} \in \Theta $. Applying Lemma \ref{lem:emp} yields the bound
\begin{align*}
  &\mathbb{P} \left( \max_{i: z_{i} \in \Theta}  \left|\int a(y) c(z_{i},y;r)  \mathbb{I}(y \in \Theta) 
      p_{\Theta,\beta}^{\ast} (y) [\sigma_{n\backslash i}(dy) - \sigma(y)\,dy]   \right|  
                   > \frac{\emin \epsilon_{\Theta,\beta} (\beta-1)}{L_{f}}  \right) \\
  & \leq 2n \exp\left( -C_2 \frac{(n-1)r^d}{A}
             \frac{ \emin^2 \epsilon_{\Theta,\beta}^{2} (\beta-1)^{2}}{\amax^2L_{f}^{2}} \right), 
\end{align*}
if inequality~(\ref{LB:ineq3}) holds.
\end{proof}

Lemma \ref{lem:lb1} shows that inequality (\ref{LB:eq1}) holds with high probability if $ p^{\ast}_{\Theta,\beta} $ 
can be bounded away from zero. We now establish a lower bound on $ p^{\ast}_{\Theta,\beta} $. 

To state the lemma that we need, some further notation is necessary. With $ \Theta := \Theta_{x,t}$ as before, 
suppose that $\eth := \min_{z\in\Theta}q_{1}(z) > 0$. Recall~$C_1$, as introduced following Assumption~E, 
and set
\eqa
    \cth &:=& \min_{z\in\Theta} \int^{1}_{0}  c_{z}(\lambda) \lambda^{d}\, d\lambda; \label{cmin}\\
 C_3 &:=& v_d\cmax (\amax L_{\sigma} + \smax L_{a}).  \label{C5}
\ena

\begin{lemma} \label{lem:LB}
Suppose that Assumptions B--E hold. Define 
\[
     q_{\Theta,\beta^{\prime},m}(z) := \left( m(t-\|z-x\|) \wedge q_{\beta^{\prime}}(z) \right).
\]
If there exists  constants $ \beta^{\prime} \in (\beta, (1-\eth)^{-1}),\ \theta_{1} \in (1,\infty) ,\ \theta_{2} \in (0,1) $ 
and $ m \in (0,\infty)  $ such that 
\begin{align}
   &(1+ \theta_{1})mr  \Le (\b'\eth + 1 - \b'); \label{ineq:L0} \\
   &L_{f} \rmax (m\vee L_q) r  \Le (\beta^{\prime} - \beta) \emin \theta_{1} mr;  \label{ineq:L1} \\
   &\frac{L_{f} (C_3 + \rmax/t) r}{\theta_{2}} + \rmax (C_1 \rmax +L_f) \theta_{1} mr 
             \Le    \emin (\b \eth + 1 - \b); \label{ineq:L2}\\
   &\frac{r}{t}   \Le \min\Bigl\{ \theta_{2},\frac1{2(2+\theta_{1})} \Bigr\};  \label{ineq:theta2}  \\
   &(C_3  + \rmax/t)r \Le \amin v_{d-1} \smin \left( \cth -2\cmax \theta_{2}\right)
               (1-\theta^{2}_{2})^{(d-1)/2}, \label{ineq:theta2c}
\end{align}
then $ q_{\Theta,\beta^{\prime},m}(z) \leq p_{\Theta,\beta}^{\ast}(z) $ for all $ z \in \Theta $. 
\end{lemma}

\begin{proof}
Suppose that 
\begin{equation}
   \frac{\beta e(z) q_{\Theta,\beta^{\prime},m}(z)}{1-q_{\Theta,\beta^{\prime},m}(z)} 
       \Le f\left( \int a(y) c(z,y;r) q_{\Theta,\beta^{\prime},m}(y)  \mathbb{I}(y \in \Theta) \s(y)\,dy \right) 
                   \label{lem:lb2:eq1}
\end{equation}
for all $ z \in \Theta $. Then $ E_{\Theta,\beta} $ maps $ \{p \in C^{+}(\Theta): q_{\Theta,\beta^{\prime},m} \leq p\} $ 
into itself. The map $ E_{\Theta,\beta} $ is compact by Assumption C. By the Schauder fixed point theorem, 
$ q_{\Theta,\beta',m} \leq p^{\ast}_{\Theta,\beta} $. We now verify that inequality (\ref{lem:lb2:eq1}) holds.

{\bf L1}.For any $ z$ such that $\|z -x\| \leq t-r$,
\[
     \int a(y) c(z,y;r) q_{\Theta,\beta^{\prime},m}(y) \mathbb{I}(y\in\Theta) \s(y)\,dy
           \Eq \int a(y) c(z,y;r) q_{\Theta,\beta^{\prime},m}(y) \s(y)\,dy. 
\]
From Lemma \ref{lem:Lips}, $ q_{\Theta,\beta^{\prime},m} $ is Lipschitz continuous with constant $ (m \vee L_q) $. Hence,
\begin{align}
    \lefteqn{f\left(  \int a(y) c(z,y;r) q_{\Theta,\beta^{\prime},m}(y)  \mathbb{I}(y \in \Theta) \s(y)\,dy \right)} \nonumber \\
     &\Eq  f\left( q_{\Theta,\beta^{\prime},m}(z) \int a(y) c(z,y;r)  \s(y)\,dy 
         + \int a(y) c(z,y;r) [q_{\Theta,\beta^{\prime},m}(y) - q_{\Theta,\beta^{\prime},m}(z)] \s(y)\,dy \right) \nonumber \\
     &\ \geq\ f\left(\rho(z) q_{\Theta,\beta^{\prime},m}(z)\right) - L_{f} \rho(z) (m\vee L_q) r. \label{lem:lb:eq1b}
\end{align}
As $ q_{\Theta,\beta^{\prime},m}(z) \leq q_{\beta^{\prime}}(z) $, we can apply Lemma \ref{lem:Fbound2} with 
inequality~(\ref{lem:lb:eq1b}) to show
\begin{align}
   \lefteqn{f\left(\int a(y) c(z,y;r) q_{\Theta,\beta^{\prime},m}(y) \mathbb{I}(y\in \Theta) \s(y)\,dy\right) 
                - \frac{\beta e(z) q_{\Theta,\beta^{\prime},m}(z)}{1-q_{\Theta,\beta^{\prime},m}(z)}}  \nonumber\\
 &\qquad \geq \frac{(\beta^{\prime} - \beta) e(z) q_{\Theta,\beta^{\prime},m}(z)}{1-q_{\Theta,\beta^{\prime},m}(z)} 
          - L_{f}\rho(z) (m\vee L_q) r.    \phantom{XXXXXXX}  \label{lem:lb2:eq2}
\end{align} 
From the definition of~$\eth$, $q_{\b'}(z) \ge (\b'\eth + 1 - \b')$ for all~$z\in\Theta$, by Lemma~\ref{lem:Fbound}. 
Set $ \Theta_{1} := \{y: \|y-x\| \leq t - \theta_{1} r \} $. Then
$
      q_{\Theta,\beta^{\prime},m}(z)\ \geq\ \theta_{1} mr 
$ 
for all $ z \in \Theta_{1} $ by inequality (\ref{ineq:L0}). Applying this lower bound to inequality~(\ref{lem:lb2:eq2}), 
we see that inequality~(\ref{lem:lb2:eq1}) holds for all $ z \in \Theta_{1} $ if inequality~(\ref{ineq:L1}) holds.

{\bf L2}. Define $ \Theta_{2} := \{y:  t  - \theta_{1}r  <\|y-x\| \leq t  - \theta_{2} r \} $. 
For any $ z \in \Theta_{2} $  and $ y $ such that $ \|y-z\| \leq r $, 
\[
     m(t-\|y-x\|) \Le m(t - \|z-x\| + \|z-y\|) \Le (1+ \theta_{1}) mr \Le q_{\beta^{\prime}}(y)
\]
by Lemma~\ref{lem:Fbound} and since $ (1+\theta_{1}) mr \leq \beta^{\prime}\eta_{\Theta} + 1 - \beta^{\prime}  $ by  
inequality~(\ref{ineq:L0}). Therefore, for any $ z \in \Theta_{2} $  and $ y \in \Theta $ such that 
$ \|y-z\| \leq r $, $ q_{\Theta,\beta^{\prime},m}(y) = m(t-\|y-x\|)$. 
For any $ z \in \Theta_{2} $  and $ y \not\in \Theta $ such that $ \|y-z\| \leq r $, we have $ m(t-\|y-x\|) \leq 0 $. 
Hence
\begin{align}
   \lefteqn{\int a(y) c(z,y;r) q_{\Theta,\beta^{\prime},m}(y)  \mathbb{I}(y\in\Theta) \s(y)\,dy} \nonumber \\
   &\quad\ \geq\ \int a(y) c(z,y;r) m(t- \|y-x\|) \s(y)\,dy \nonumber \\
   &\quad\Eq  m(t-\|z-x\|) \rho(z) + m \int a(y) c(z,y;r) \left[\|z-x\| - \|y - x\|\right]\s(y)\,dy. \label{lem:L3:eq1}
\end{align}
Let $ \gamma(x,y,z) $ be the angle formed between the vectors $ x-z $ and $ y-z $. By the cosine rule
\begin{align*}
   \lefteqn{\|z-x\| - \|y-x\| - \|z-y\| \cos \gamma(x,y,z)} \\
    & \Eq \|z-x\| \left(1 - \left(1+ \left(\frac{\|z-y\|}{\|z-x\|}\right)^{2} 
             - 2 \left(\frac{\|z-y\|}{\|z-x\|}\right) \cos\gamma(x,y,z)\right)^{1/2} \right)  - \|z-y\|\cos\gamma(x,y,z) .
\end{align*}
Let $ h(u) = (1 + u^{2} -2u\cos \gamma)^{1/2} $. Taking a Taylor expansion about $ 0 $ gives 
$ h(u) = 1-u\cos\gamma + \frac{1}{2} u^{2} h^{\prime\prime}(\tilde{u}) $ for some $ \tilde{u} \in (0,u)$. Therefore,
\begin{align*}
\lefteqn{\|z-x\| - \|y-x\| - \|z-y\| \cos \gamma(x,y,z)} \\
    &\ \geq\ -\frac{\|z-y\|^{2}}{2\|z-x\|} 
          \sup_{\xi \in (0, \frac{\|z-y\|}{\|z-x\|})} \left(1 + \xi^{2} - 2\xi \cos\gamma(x,y,z)\right)^{-1/2} \\
    &\ \geq\ -\frac{r^{2}}{2\|z-x\|} \left(1 - \frac{2r}{\|z-x\|}\right)^{-1}.
\end{align*} 
Noting that $ \|z-x\| \geq t - \theta_{1} r $ and substituting this bound into~(\ref{lem:L3:eq1}) gives
\begin{align*}
     \lefteqn{\int a(y) c(z,y;r) q_{\Theta,\beta^{\prime},m}(y) \mathbb{I}(y\in\Theta) \s(y)\,dy} \\
    &\quad \geq\  q_{\Theta,\beta^{\prime},m}(z) \rho(z) + m \int a(y) c(z,y;r) \|z-y\| \cos\gamma(x,y,z) \s(y)\,dy  
          - \frac{mr^{2}\rho(z)}{2(t-(2+\theta_{1})r)},
\end{align*}
if $ t > (2+\theta_{1})r$; but this follows from inequality~(\ref{ineq:theta2}), 
which gives $t-(2+\theta_{1})r \geq t/2 $. Now, from Assumption~C,
\begin{align}
   \lefteqn{\int a(y) c(z,y;r) q_{\Theta,\beta^{\prime},m}(y) \mathbb{I}(y\in\Theta) \s(y)\,dy} \nonumber \\
    &\quad \geq\  q_{\Theta,\beta^{\prime},m}(z) \rho(z) + m a(z) \sigma(z) \int  c(z,y;r) \|z-y\| \cos\gamma(x,y,z) dy 
                        \nonumber \\
   & \ \qquad + m a(z) \int c(z,y;r) \|z-y\| \cos\gamma(x,y,z) [\sigma(y) - \sigma(z)] dy  \nonumber \\
   & \  \qquad + m \int [a(y) -a(z)] c(z,y;r) \|z-y\| \cos\gamma(x,y,z) \sigma(y) dy   -\frac{mr^{2}\rho(z)}{t}  \nonumber\\
   &\quad \geq\ q_{\Theta,\beta^{\prime},m}(z) \rho(z) + m a(z) \sigma(z) \int  c(z,y;r) \|z-y\| \cos\gamma(x,y,z) dy 
                 - (C_3 + \rmax/t) mr^{2}. \label{lem:L2:eq0}
\end{align}
By the radial symmetry of $ c(z,y;r) $, $ \int  c(z,y;r) \|z-y\| \cos\gamma(x,y,z) dy = 0 $. Thus we deduce that
\begin{align*}
     f\left( \int a(y) c(z,y;r) q_{\Theta,\beta^{\prime},m}(y) \s(y)\,dy \right)  
       \ \geq\  f\left(q_{\Theta,\beta^{\prime},m}(z) \rho(z) \right) - L_{f} (C_3 + \rmax/t)  mr^{2} .
\end{align*}
Therefore, applying Lemma \ref{lem:Fbound2} and noting that $ f(x) \geq L_f x - C_1 x^{2}  $ gives
\begin{align}
    \lefteqn{f\left( \int a(y) c(z,y;r) q_{\Theta,\beta^{\prime},m}(y) \s(y)\,dy \right)  
          - \frac{\beta e(z) q_{\Theta,\beta^{\prime},m}(z)}{1-q_{\Theta,\beta^{\prime},m}(z)}} \nonumber \\
   &\ \geq\   f\left(q_{\Theta,\beta^{\prime},m}(z) \rho(z) \right) - \frac{\beta e(z) q_{\Theta,\beta^{\prime},m}(z)}
             {1-q_{\Theta,\beta^{\prime},m}(z)}  - L_{f}(C_3 + \rmax/t) mr^{2}  \nonumber \\
   &\ \geq\ \frac{q_{\Theta,\beta^{\prime},m}(z)}{1-q_{\Theta,\beta^{\prime},m}(z)}  
             \left( L_f \rho(z) (1-q_{\Theta,\beta^{\prime},m}(z))  
       - \beta e(z)  -\frac{L_{f}(C_3 + \rmax/t) mr^{2}}{ q_{\Theta,\beta^{\prime},m}(z)} 
                        - C_1 \rho^{2}(z) q_{\Theta,\beta^{\prime},m}(z) \right) \nonumber \\
  &\ \geq\ \frac{q_{\Theta,\beta^{\prime},m}(z)}{1-q_{\Theta,\beta^{\prime},m}(z)}  \left( L_f \rho(z)  
      - \beta e(z) - \frac{L_{f}(C_3 + \rmax/t)  mr^{2}}{ q_{\Theta,\beta^{\prime},m}(z)} 
            - \rho(z)  q_{\Theta,\beta^{\prime},m}(z)  ( C_1\rho(z) + L_f)\right) 
             \label{lem:L2:eq3}
\end{align}

We now need a lower bound on $ L_f \rho(z) - \beta e(z) $. 
% By definition, $ q_{\beta}(z) $ satisfies 
% $ (1-q_{\beta}(z)) f^{\prime}(\rho(z)q_{\beta}(z)) = \beta e(z) q_{\beta}(z) $. 
By Assumption~E, we have
$$ 
     (1 - q_{\beta}(z)) L_f \rho(z) q_{\beta}(z)\ \geq\ \beta e(z) q_{\beta}(z) .  
$$
Hence, because $L_f \rho(z) > e(z)$ whenever $q_1(z) > 0$, we deduce that
$$ 
     L_f \rho(z) - \beta e(z)\ \geq\ L_f \rho(z) q_{\beta}(z)\ \geq\ e(z) q_{\beta}(z) .
$$

This, together with the lower bound on $ q_{\beta}(z) $ from Lemma \ref{lem:Fbound}, gives 
$ L_f \rho(z) - \b e(z) \geq \emin (\b \eth +1 - \b) $.
As $ \theta_{2} mr \leq q_{\Theta,\beta^{\prime},m}(z) \leq \theta_{1} mr  $ for all $ z \in \Theta_{2} $ 
we see that the right hand side of inequality (\ref{lem:L2:eq3}) is positive if inequality~(\ref{ineq:L2}) holds. 
Hence, inequality~(\ref{lem:lb2:eq1}) holds for all $ z \in \Theta_{2} $.

{\bf L3}. Define $ \Theta_{3} := \{y:  t - \theta_{2}r <\|y-x\| \leq t\} $. As in {\bf L2}, for any $ z \in \Theta_{3} $  
and $ y $ such that  $ \|y-z\| \leq r $, we have $ q_{\Theta,\beta^{\prime},m}(y) = m(t-\|y-x\|)$.  Following  
inequality~(\ref{lem:L2:eq0}) in {\bf L2}, 
\begin{align*}
  \lefteqn{\int a(y) c(z,y;r) q_{\Theta,\beta^{\prime},m}(y) \mathbb{I}(y\in\Theta) \s(y)\,dy - m(t-\|z-x\|)\rho(z)} \\
    &\ \geq\   m a(z) \sigma(z) \int c(z,y;r) \|z-y\| \cos\gamma(x,y,z) \mathbb{I}(y \in \Theta) dy  
                  - (C_3 + \rmax/t)  mr^{2} .
\end{align*}
As $ \theta_{2} < 1 $, let~$w$  be a point of intersection of the ball $B_z(r)$ with~$\partial\Theta$, 
and let $\phi := \gamma(x,w,z)$.
Applying the change of variable $\lambda(y) =  r^{-1} \| z - y\| $ and $ \omega(y) = \gamma(x,y,z) $ yields
\[
       \int_{\{ y : -\phi \leq \omega(y) \leq \phi \}} c(z,y;r) \|z-y\| \cos\omega(y) \mathbb{I}(y \in \Theta) dy 
         \ \geq\ \kd r\left(\int^{1}_{0}  c_{z}(\lambda) \lambda^{d}\, d\lambda \right) (\sin\phi)^{d-1}.
\]
 It remains to determine a lower bound for the integral
\eq\label{negative-contn}
     \int_{\{ y : \gamma(x,y,z) \in [-\pi,-\phi) \cup (\phi,\pi]\}} c(z,y;r) \|z-y\| \cos\gamma(x,y,z) 
               \mathbb{I}(y \in \Theta)\, dy .
\en

The region of integration is included in a cylinder of height
$(t-\|z-x\|) + r\max\{\cos\phi,0\}$ and  radius $r\sin\phi$. 
As $ \phi $ is determined by the intersection of two circles,
\begin{equation*}
   \cos\phi  \Eq r^{-1} \left( \|z-x\|- \frac{\|z-x\|^{2} -r^{2} + t^{2}}{2\|z-x\|} \right)  
     \Eq \frac{\|z-x\|-t}{r} + \frac{r[1 - ((t-\|z-x\|)/r)^{2}]}{2 \|z-x\|}. \label{lem:L3:eq6}
\end{equation*}
The function $ x + r(1-x^2)/(2(t+rx)) $ is increasing when $ t > r $, and  so 
\begin{equation}
 -\theta_{2} \Le  \cos\phi \Le \frac{r}{2t}.  \label{ineq:last}
 \end{equation}
Hence, for $z \in \Theta_3$, the volume of integration cannot exceed
$$  
     v_{d-1}(r\sin\phi)^{d-1}(\th_2 r + r^2/t) \Le 2 v_{d-1} (r\sin \phi)^{d-1} \th_{2} r ,
$$ 
by inequality (\ref{ineq:theta2}). The largest negative value of the integrand is bounded below by $-\cmax r^{-d+1}$.  
Hence this integral is bounded below by $  -2\cmax v_{d-1}r (\sin\phi)^{d-1}\th_2 $. This leads to the lower bound
\begin{align}
  \lefteqn{\int a(y) c(z,y;r) q_{\Theta,\beta^{\prime},m}(y) \mathbb{I}(y\in\Theta) \s(y)\,dy - m(t-\|z-x\|)\rho(z)} \nonumber \\
    &\quad \geq\   m r a(z) \sigma(z) v_{d-1} \left( \cth -  2\cmax \th_2 \right)(\sin\phi)^{d-1}   
                   -  (C_3 + \rmax/t)  mr^2.      \label{lem:L3:eq7} 
\end{align}
From inequalities (\ref{ineq:theta2}) and~(\ref{ineq:last}), $ (\sin \phi)^{d-1} \geq (1-\theta^{2}_{2})^{(d-1)/2} $. 
Applying inequality~(\ref{ineq:theta2c}), we see that the right-hand side of~(\ref{lem:L3:eq7}) is positive. Therefore, 
$$
   f(\rho(z)q_{\Theta,\beta^{\prime},m}(z)) 
     \Le f\left(\int a(y) c(z,y;r) q_{\Theta,\beta^{\prime},m}(y) \mathbb{I}(y\in\Theta) \s(y)\,dy\right). 
$$ 
Lemma~\ref{lem:Fbound2}, with $\t = \rho(z)$ and $\nu = \beta'e(z)$, then implies that inequality (\ref{lem:lb2:eq1}) 
holds for all $ z \in \Theta_{3} $. 
Hence,  $ q_{\Theta,\beta^{\prime},m}(z) \leq p_{\Theta,\beta}^{\ast}(z) $ for all $ z \in \Theta $.

\end{proof}

\begin{lemma} \label{lem:ULB}
Suppose that Assumptions B--E hold, that $ \min_{z\in\Theta}q_{1}(z) =: \eta_{\Theta} > 0$ and 
that $ \beta \in (1,1+ \eth/2) $.  Assume that 
\begin{align}
   L_{q} r &\Le \frac{\eth^{2} \emin^{2}}{32 L_{f} \rmax^{2}(C_1 \rmax + L_f)};  \label{lem:ULB:ineq1} \\
   \frac{r}{t} &\Le \left( \frac{ \cth}{4 \cmax }  \wedge \frac{1}{\sqrt{2}} 
                \wedge \frac{\eth \emin}{8L_{f} \rmax + 4 \eth \emin}\right);  \label{lem:ULB:ineq2} \\
   (C_3 + \rmax/t)r &\Le  \left\{ \amin \smin \cth v_{d-1} 2^{-(d+3)/2}\right\} 
          \wedge \left\{ \frac{\eth \emin}{4L_{f}}  
      \left( \frac{\cth}{4 \cmax} \wedge \frac{1}{\sqrt{2}} \right) \right\}, \label{lem:ULB:ineq3} 
\end{align}
and define
\eq\label{C7-def}
    C_4 \Def \left(1 \wedge \frac{f(\amin \smin)}{2^{(d+1)/2} \emax}\right)  \left( \frac{ \cth}{4 \cmax} 
        \wedge \frac{1}{\sqrt{2}} \right)  \frac{\emin^{2}}{32 L_{f} \rmax^{2}(C_1 \rmax + L_f)}  .
\en
Then $  p_{\Theta,\beta}^{\ast}(z) \geq C_4  \eth^{2} $ for all $ z \in \Theta$. 
\end{lemma}

\begin{proof}
We begin by showing that the above inequalities are sufficient for the inequalities of Lemma \ref{lem:LB} to hold,
for suitable choices of $\b',\th_1,\th_2$ and~$m$. 
Set $ \beta' = \frac{1}{2(1-\eth)} + \frac{\beta}{2} $. Then 
\begin{equation}
   \beta^{\prime} - \beta  \Eq \frac{1}{2(1-\eth)} - \frac{\beta}{2}  = \frac{\beta\eth + 1 - \beta}{2(1-\eth)}
   \ \geq\ \frac{\eth (1+\eth)}{4(1-\eth)}\ \geq\ \frac{\eth}{4}, \label{proof:ULB:eq1}
\end{equation}
and
\begin{equation}
    \beta^{\prime} \eth + 1 - \beta^{\prime} \Eq \frac{1}{2} (\beta \eth + 1 - \beta)
           \ \geq\ \frac{\eth}{4}. \label{proof:ULB:eq2}
\end{equation}
Set 
\begin{align*}
    \theta_{1} &\Def \frac{4 L_{f} \rmax}{\eth \emin}; \\
    mr &\Def \frac{\eth^{2} \emin^{2}}{32 L_{f} \rmax^{2}(C_1 \rmax + L_f)};  \\
    \theta_{2} &\Def \left( \frac{ \cth}{4 \cmax } \wedge \frac{1}{\sqrt{2}} \right). 
\end{align*}
Since
\[
   \frac{L_{f} \rmax}{\eth \emin}\ \geq\ \frac{ L_f \rho(z)}{\eth e(z)}\ \geq\ \frac{1 + q_{1}(z)}{\eth} 
    \ \geq\ \frac{1 + \eth}{\eth}\ \geq\ 1,
\]
it follows that $ \theta_{1} > 1 $. This, together with inequality~(\ref{proof:ULB:eq2}), implies that 
inequality~(\ref{ineq:L0}) is satisfied if $ 8 \theta_{1} mr \leq \eth $. This is indeed the case, since, from the choices of $\th_1$ and~$mr$, we have
\[
     8\th_1 mr \Le \eth\,\frac{\emin}{L_f\rmax} \Le \eth,
\]
because $L_f \rmax/\emin > 1$ if $\eth > 0$.

Then $ L_{q} r \leq mr $,
by inequality~(\ref{lem:ULB:ineq1}),  so inequality~(\ref{ineq:L1}) simplifies to give 
$ L_{f} \rmax \leq (\beta^{\prime} - \beta) \emin \theta_{1} $; and this is seen to hold, by inequality~(\ref{proof:ULB:eq1})
and the choice of~$\th_1$.
The choices of $ \theta_{1} $ and $ \theta_{2} $, together with inequality~(\ref{proof:ULB:eq2}), show further that 
inequality~(\ref{ineq:L2}) is implied by inequality~(\ref{lem:ULB:ineq3}), and that 
inequality~(\ref{ineq:theta2}) is implied by inequality~(\ref{lem:ULB:ineq2}). 
Finally,  the choice of~$\theta_{2} $ shows that inequality~(\ref{ineq:theta2c}) follows from 
inequality~(\ref{lem:ULB:ineq3}).
Thus, inequalities (\ref{ineq:L0})--(\ref{ineq:theta2c}) in Lemma~\ref{lem:LB} hold.

Take $ \Theta_{1},\ \Theta_{2} $ and $ \Theta_{3} $ as defined in the proof of Lemma \ref{lem:LB}. 
On $ \Theta_{1} \cup \Theta_{2} $, 
\[
    p^{\ast}_{\Theta,\beta}(z)\ \geq\ q_{\Theta,\beta^{\prime},m}(z)\ \geq\ \theta_{2} mr. 
\]
For $ z \in \Theta_{3} $, note that 
$ q_{\Theta,\beta^{\prime},m} \leq E_{\Theta,\beta}(q_{\Theta,\beta^{\prime},m}) \leq p^{\ast}_{\Theta,\beta} $ and that
\[
    E_{\Theta,\beta}(q_{\Theta,\beta^{\prime},m})(z)
      \ \geq\ \frac{f\left(\int a(y) c(z,y;r) q_{\Theta,\beta^{\prime},m}(y) \mathbb{I}(y\in \Theta) \s(y)\,dy\right)}
                          {\beta e(z)}.
\]

Now inequalities (\ref{lem:L3:eq7}) and~(\ref{lem:ULB:ineq3}) imply that
$$
     f\left(\int a(y) c(z,y;r) q_{\Theta,\beta^{\prime},m}(y) \mathbb{I}(y\in\Theta) \s(y)\,dy\right)
        \ \ge\ f\left(mr\,  \frac{\amin \smin \cth}{2^{(d+3)/2} \cmax } \right) 
          \ \ge\  f\left(\th_2 mr\,  \frac{\amin \smin}{2^{(d-1)/2} } \right).
$$ 
Then, from Assumption~E, we have $f(ab) \ge bf(a)$ if $0 \le b \le 1$.
Now $\th_2 < 1$, $2^{(d-1)/2} \ge 1$ and $mr \le 1/32$, because $0 < \eth \le 1$ and $L_f \rmax/\emin > 1$, 
so we conclude that
\[
    p^{\ast}_{\Theta,\beta}(z)\ \geq\  \frac{f\left(\amin \smin\right)}{2^{(d-1)/2}}\, \theta_{2} mr\,\frac1{\beta\emax}
      \ \geq\  \frac{f\left(\amin \smin\right)}{2^{(d+1)/2}\emax}\, \theta_{2} mr
\]
for all $ z \in \Theta_{3} $. 

Combining this with the lower bound on $ q_{\Theta,\beta^{\prime},m} $ for 
$ z \in \Theta_{1} \cup \Theta_{2} $ gives the uniform lower bound.
\end{proof}

\begin{lemma} \label{lem:localLB}
Suppose that Assumptions B--E hold, that $ \min_{z\in\Theta} q_{1}(z) =: \eta_{\Theta} > 0$ and that 
$1 < \beta < \beta^{\prime} < 1+ \eth/2 $. 
Assume that, in addition to inequalities (\ref{lem:ULB:ineq2})--(\ref{lem:ULB:ineq3}),
\begin{align}
   L_{q} r &\Le \frac{\emin^{2} \eth (\beta' - \beta)}{4 \rmax^{2} L_{f} (C_1\rmax + L_f)}; \label{lem:localLB:ineq1} \\
   \frac{r}{t} &\Le \frac{\emin(\beta' - \beta)}{6L_{f} \rmax}.  \label{lem:localLB:ineq2}
\end{align}
Then, choosing~$m$ as in Theorem~\ref{thm:LB} so that
\[
   mr \Eq \frac{\emin^{2} \eth (\beta' - \beta)}{4 \rmax^{2} L_{f} (C_1 \rmax + L_f)},
\]
it follows that $  p_{\Theta,\beta}^{\ast}(z) \geq q_{\Theta,\beta^{\prime},m}(z) $ for all $ z \in \Theta$.
\end{lemma}

\begin{proof}
We show that the above inequalities are sufficient for the inequalities of Lemma~\ref{lem:LB} to hold,
with suitable choices of $\th_1$ and~$\th_2$.  Set 
\begin{align*}
    \theta_{1} &\Def \frac{L_{f} \rmax}{\emin(\beta' - \beta)},
%  mr & := \frac{e^{2}_{min} \eth (\beta' - \beta)}{4 \rmax^{2} L_{f} (C_1 \rmax + L_f)}.
\end{align*}
and choose~$ \theta_{2} $ as in Lemma~\ref{lem:ULB}. Note that $ \theta_{1} > 1 $, since 
$ \beta'-\beta \leq 1\leq L_{f} \rmax/\emin $.
Since $ \beta' \eth +1 - \beta' \geq \eth/2 $, inequality~(\ref{ineq:L0}) holds if $ 4mr \leq \eth $;
but this is true with the above choice of $ mr $, because $ (\beta' - \beta) < 1$ and $L_f \rmax/\emin > 1$. 

From inequality (\ref{lem:localLB:ineq1}), $ L_{q} r \leq mr $, and so inequality~(\ref{ineq:L1}) simplifies to give 
$ L_{f} \rmax \leq (\beta^{\prime} - \beta) e_{\min} \theta_{1} $, which holds with equality for~$\theta_{1} $
as chosen. To show that inequality~(\ref{ineq:L2}) holds, we first note that 
$ 4 \rmax (C_1 \rmax + L_f)  \theta_{1} mr \leq \emin \eth $. Therefore, 
inequality (\ref{ineq:L2}) holds if $ 4 L_{f} (C_3 + \rmax/t) r \leq \theta_{2} \emin \eth $, which follows 
from inequality (\ref{lem:ULB:ineq3}). 

The second part of inequality~(\ref{ineq:theta2}) holds by (\ref{lem:localLB:ineq2}) and because
\[
    2(2+\theta_{1}) \Eq \frac{4 \emin (\beta' - \beta) + 2 L_{f} \rmax}{\emin(\beta'-\beta)} 
      \Le \frac{6L_{f}\rmax}{\emin(\beta'-\beta)},
\]
again since $ \beta'-\beta \leq 1\leq L_{f} \rmax/\emin $. 

Finally, with~$\theta_{2} $ chosen as in Lemma~\ref{lem:ULB}, 
inequality~(\ref{ineq:theta2c}) follows from~\Ref{lem:ULB:ineq3} as in Lemma~\ref{lem:ULB},
and the first part of inequality~(\ref{ineq:theta2}) follows from~\Ref{lem:ULB:ineq2}.

\end{proof}

\begin{proof}[Proof of Theorem \ref{thm:LB}]
First note that inequality (\ref{LB:ineq3}) of Lemma \ref{lem:lb1} holds with 
$ \epsilon_{\Theta,\beta} = C_4 \eth^{2} $, since 
\[
 C_{4} \Le \frac{\emin^{2}}{32\sqrt{2} L_{f}^{2} \rmax^{2}} \Le \frac1{8\sqrt2}
\] 
when $ L_{f} \rmax / \emin > 1/2 $, and hence
\[
     C_4 \eth^2 (\beta-1) \Le C_4 \frac{\eth^3}2 \Le \frac1{16\sqrt2} \ <\ \frac12 \Le \frac{L_f\rmax}{\emin}.
\]

Now combine Lemmas \ref{lem:lb1}, \ref{lem:ULB} and \ref{lem:localLB}.
\end{proof}

\begin{proof}[Proof of Corollary \ref{cor1}]
By Theorem \ref{thm:UB}, $ p_{i,n}^{\ast} \leq p^{+}_{\alpha_{1},\alpha_{2}}(z_{i}) $ for all $ i=1,\ldots,n $ 
with high probability. Taking $ \alpha_{2} = 1- \eth $, we note that $q_{\alpha_{1}}(z) \geq q_{1}(z) \geq \eth $,
and so $ p_{\alpha_{1},\alpha_{2}}^{+} (z) = q_{\alpha_{1}}(z) $ for all $ z \in \Theta $. Therefore,
\begin{equation}
    p_{n,i}^{\ast} - q_{1}(z_{i}) \Le q_{\alpha_{1}}(z_{i}) -q_{1}(z_{i}), \label{cor:eq1}
\end{equation}
for all $ z_{i} \in \Theta $, with high probability. 

Note that $ q_{\Theta,\beta',m}(z) = q_{\beta'}(z) $ for all $ z \in \Theta_{m} $. By Theorem \ref{thm:LB} for all $ z_{i} \in \Theta_{m} $ 
\begin{equation}
    p_{n,i}^{\ast} - q_{1}(z_{i})\ \geq\ q_{\beta'}(z_{i}) - q_{1}(z_{i}) \label{cor:eq2}
\end{equation}
with high probability. Inequalities (\ref{cor:eq1}) and (\ref{cor:eq2}) imply that, for all $ z_{i} \in \Theta_{m} $,
\[
    \left| p_{n,i}^{\ast} - q_{1}(z_{i}) \right| \Le q_{\alpha}(z_{i}) - q_{\beta'}(z_{i}) 
\]
with high probability. As in the proof of Lemma \ref{lem:Lips} 
\[
    \frac{\partial q_{\alpha}(z)}{\partial \alpha} \Le \frac{-e(z)}{\alpha e(z) + f(\rho(z) q_{\alpha}(z)}.
\]
Hence,
\[
   q_{\alpha_1}(z) - q_{\beta'}(z) \Le \int_{\alpha_1}^{\beta'} \frac{e(z)\, du}{ue(z) + f(\rho(z)q_{u}(z))} 
          \Le \a_1^{-1}(\beta'-\alpha_1).
\]
\end{proof}

\begin{proof}[Proof of Corollary \ref{cor2}]
The corollary follows from Corollary~\ref{cor1}, with appropriate choices of $ \a_1, \a_2, \b $ and $ \b' $. First note that $r_n^{1-\g_1}\phi_n \le c_1r_n^{\g_1}/4$ for all~$n$ sufficiently large, if  $r_n^{1-2\g_1}\phi_n \to 0$.  Thus we can take $1-\a_2 = \eta_{\Omega_n}$, $ 1- \a_1 = r_n^{1-\g_1} \phi_n $ and $\b'-\b = \b-1 = r_n^{1-\g_1}\phi_n$, and satisfy $\a_1 \ge \a_2$ and $\b'-1 \le \eta_{\Omega_n}/2$ for all~$n$ sufficiently large.  With these choices of $\a_1$ and~$\a_2$, 
inequality \Ref{UB:ineq1} of Theorem~\ref{thm:UB} holds for all~$ n $  sufficiently large; the choices of $\b$ and~$\b'$ show that inequality \Ref{essential-conditions} holds,  fulfilling the conditions of Theorem~\ref{thm:LB}.  Then the probabilities in Corollary~\ref{cor1} converge to~$1$, as required, in view of~\Ref{cor2-1} and~\Ref{r-smooth}, and $\a_1^{-1}(\b'-\a_1) = O(r_n\phi_n)$  and $ m \asymp \phi_{n}$.
\end{proof}

\end{document}